\documentclass[11pt,a4paper,reqno]{amsart}
\usepackage[pagebackref]{hyperref}
\usepackage[utf8]{inputenc}
\usepackage[T1]{fontenc}
\usepackage{amssymb,amsmath,amsthm,amscd,mathrsfs,amsfonts,latexsym,enumitem,array} 
\usepackage{mathtools}
\usepackage[all]{xy}
\usepackage[dvipsnames]{xcolor}
\usepackage[top=2.7cm, bottom=2.7cm, left=2.7cm,right=2.7cm]{geometry}
\usepackage{orcidlink}
\usepackage{cite,mdwlist}
\usepackage{tikz}
\usepackage{BOONDOX-calo}
\usepackage{dsfont}

\newtheorem{thm}{Theorem}[section]
\newtheorem{mainthm}{Theorem}
\newtheorem{prop}[thm]{Proposition}
\newtheorem{lemma}[thm]{Lemma}
\newtheorem{cor}[thm]{Corollary}

\theoremstyle{definition}
\newtheorem{defi}[thm]{Definition}

\theoremstyle{remark}

\newtheorem{rmk}[thm]{Remark}
\newtheorem{ex}[thm]{Example}

\newtheorem{conse}[thm]{Consequence}
\newtheorem{obj}[thm]{Objective}

\DeclareMathOperator{\Z}{\mathds{Z}}

\DeclareMathOperator{\N}{\mathds{N}}
\DeclareMathOperator{\G}{\mathds{G}}

\DeclareMathOperator{\Hom}{Hom}

\DeclareMathOperator{\Pic}{Pic}
\DeclareMathOperator{\rk}{rk}
\DeclareMathOperator{\M}{M}
\DeclareMathOperator{\Spl}{Spl}

\DeclareMathOperator{\Ext}{Ext}

\DeclareMathOperator{\discr}{discr}

\DeclareMathOperator{\Supp}{Supp}
\DeclareMathOperator{\End}{End}

\DeclareMathOperator{\ch}{ch}
\DeclareMathOperator{\id}{id}

\DeclareMathOperator{\coker}{coker}

\DeclareMathOperator{\Ker}{Ker}

\DeclareMathOperator{\C}{\mathds{C}}

\newcommand{\cH}{\mathcal{H}}
\newcommand{\cE}{\mathcal{E}}
\newcommand{\cK}{\mathcal{K}}
\newcommand{\cL}{\mathcal{L}}
\newcommand{\cR}{\mathcal{R}}
\newcommand{\cF}{\mathcal{F}}

\newcommand{\cG}{\mathcal{G}}
\newcommand{\cV}{\mathcal{V}}

\newcommand{\cN}{\mathcal{N}}
\newcommand{\cS}{\mathcal{S}}

\newcommand{\cI}{\mathcal{I}}

\newcommand{\cQ}{\mathcal{Q}}

\newcommand{\cO}{\mathcal{O}}
\newcommand{\cHom}{\mathcal{H}om}
\newcommand{\RHom}{\mathbf{R}{\mathrm{Hom}}}

\newcommand{\Db}{\mathbf{D}^b}

\newcommand{\bT}{\mathbf{T}}
\newcommand{\be}{\mathbf{e}}
\newcommand{\RcHom}{\mathbf{R}{\mathcal{H}om}}

\newcommand{\cExt}{\mathcal{E}xt}
\newcommand{\vv}{\mathbf{v}}
\newcommand{\s}{\mathbf{s}}

\DeclareMathOperator{\im}{Im}

\DeclareMathOperator{\HH}{H}

\DeclareMathOperator{\Bir}{Bir}

\newcommand{\ZZ}{\mathds Z}

\newcommand{\VV}{\mathds V}

\newcommand{\NN}{\mathds N}
\newcommand{\PP}{\mathds P}
\newcommand{\GG}{\mathds G}

\newcommand{\mono}{\hookrightarrow}

\newcommand{\eq}[1][r]
{\ar@<-3pt>@{-}[#1]
	\ar@<-1pt>@{}[#1]|<{}="gauche"
	\ar@<+0pt>@{}[#1]|-{}="milieu"
	\ar@<+1pt>@{}[#1]|>{}="droite"
	\ar@/^2pt/@{-}"gauche";"milieu"
	\ar@/_2pt/@{-}"milieu";"droite"}

\newcommand{\incl}[1][r]
{\ar@<-0.2pc>@{^(-}[#1] \ar@<+0.2pc>@{-}[#1]}

\newcommand{\correct}[1]{{\color{black}#1}}

\allowdisplaybreaks[1]


\thanks{G. M. has been financed by the Marco Brunella grant of
	Burgundy University and the PRCI SMAGP
        (ANR-20-CE40-0026-01).
        D. F. was partially supported by ANR FanoHK
        ANR-20-CE40-0023, SupToPhAG/EIPHI ANR-17-EURE-0002, Bridges
        ANR-21-CE40-0017, JSPS S24043.}

\begin{document}

\title[Involutions on moduli spaces of sheaves on K3
surfaces]{Anti-symplectic involutions on moduli spaces of
  sheaves on K3 surfaces via auto-equivalences}

\author[D. Faenzi]{Daniele Faenzi} %
\address{Université Bourgogne Europe, CNRS, IMB UMR 5584, F-21000 Dijon, France}
\email{daniele.faenzi@u-bourgogne.fr}

\author[G. Menet]{Gr\'egoire Menet} %
 \address{Gr\'egoire Menet, Académie militaire de Saint-Cyr Coëtquidan,
 CReC Saint-Cyr (Centre de recherche de l'académie militaire de Saint-Cyr),
 56380 Guer, France.}
 \email{gregoire.menet@ac-amiens.fr}

\author[Y. Prieto--Monta\~{n}ez]{Yulieth Prieto--Monta\~{n}ez} %
\address{Pontificia Universidad Católica de Chile, Campus San Joaquín, Avenida Vicuña Mackenna 4860, Santiago de Chile, Chile} %
\email{yulieth.prieto@uc.cl}

\maketitle
\begin{abstract}
  We provide new examples of anti-symplectic \correct{birational} involutions
  on moduli spaces of stable sheaves on K3
  surfaces. These involutions are constructed through
  (anti) auto-equivalences of the bounded derived
  category of coherent sheaves on K3 surfaces
  arising from spherical bundles.
  We analyze these induced maps in the moduli space,
  imposing restrictions on the Mukai vector and
  considering the preservation of stability conditions.
  Our construction extends and unifies classical examples,
  such as the Beauville involutions, Markman-O'Grady
  reflections and a more recent construction by
  Beri-Manivel.
\end{abstract}

\section{Introduction}

Irreducible holomorphic symplectic manifolds (IHS) are
significant objects in algebraic geometry, notably as they play
a fundamental part in the Bogomolov decomposition Theorem
\cite{Bogomolov}.
When examining these manifolds, a natural
question arises regarding their symmetries.
Symmetries and, more specifically, involutions 
have been analyzed recently under several perspectives, for instance
derived categories (see \cite{hasset-tschinkel}) and the study of
fixed loci, see \cite{flapan-macri-ogrady-sacca:I, flapan-macri-ogrady-sacca:II}.

About classification,
determining the birational automorphism group of these manifolds is a
significant problem. A classical tool for addressing this
question is the global Torelli theorem (see \cite{Verbitsky} and
\cite[Theorem 1.3]{Markman}).
For instance, Beri and Cattaneo recently
classified birational automorphisms on the Hilbert scheme $X^{[n]}$ of
$n$ points on a K3 surface $X$ with Picard number one and genus $g \ge
3$ in terms of minimal solutions of Pell's equation
$z^2-(g-1)(n-1)y^2=1$. Their result implies that, under certain
arithmetic conditions, $\Bir(X^{[n]})\simeq \Z/2\Z$,
see \cite[Theorem 1.1]{Beri}. 

However, 
constructing birational
automorphisms explicitly or geometrically remains a challenging
problem for Hilbert schemes of points, let alone for more general
moduli space $\M(\vv)$.

In this paper, we propose an explicit construction that yields
involutions on the moduli spaces \correct{$\M(\vv)$} of stable sheaves on $X$ with
a given Mukai vector $\vv$, relying on the twist along a spherical bundle
$\cS$, under certain restrictions on $\vv$ and on the Mukai
vector of $\cS$, see Theorem \ref{MainIntro} below.
In particular, this gives birational involutions on the Hilbert scheme
$X^{[n]}$ whenever $g\equiv 2\mod 4$ and $n=\frac{g+2}{4}$
(see Remark \ref{Hilbremark}).
The construction recovers and gives a unified framework to several involutions already known in the literature, such as:
\begin{itemize}
\item
  the Beauville involutions \cite{BeauvilleRemark} (see Example \ref{Beauvilleinv});
\item
  the Markman-O'Grady involutions \cite{MarkmanMain} (see Section \ref{MarkmanInv});
\item
  the Beri-Manivel involution \cite{beri2022birational} (see Example
  \ref{BeriExample}), see also \cite{beri-manivel:more-birational}.
\end{itemize}

Let us describe briefly our approach, based on
(anti)-auto-equivalences within the bounded derived category
$\Db(X)$ of coherent sheaves on $X$. Let $\cS \in
\Db(X)$ be a spherical object and
$\bT_{\cS}$ its spherical twist; see Section
\ref{constru} for more details.
Given a line bundle $\cL$ on $X$,
we introduce the following contravariant endofunctor of $\Db(X)$, actually
an anti-auto-equivalence:
\[
  \cE \mapsto \Phi_{\cS,\cL}^p(\cE) = \RcHom(\bT_{\cS}(\cE),\cL[p]).
\]
For a fixed $d \in \mathds{Z}$, we set
\[\Phi_{\cS,d}=\Phi_{\cS,\cO_X(d)}^1.\]
The main observation is that, in some cases of interest,
$\Phi_{\cS,d}$ induces an involution on moduli spaces
of stable sheaves on $X$ with a well-chosen Mukai vector.
Well-chosen here means that only very few choices of the Mukai vector
$\s$ \correct{of the spherical bundle $\cS$} and \correct{of} the twist $d$ will work and that, $\s$ and $d$ being fixed, we have a range of good
choices for $\vv$ to get involutions on $\M(\vv)$.
The two relevant cases for this purpose are the following (see Lemmas \ref{d} and \ref{rk1v}):
\begin{itemize}
\item
  $d=0$ and $\cS=\cO_X$. 
\item
  $d=1$ and $\cS$ is a stable spherical bundle of Mukai vector $(2,H,\frac{g}{2})$. 
\end{itemize}

The first case aligns with Markman-O'Grady involutions.
In this case we just recover these well-known involutions and we only
point out that their
regularity is proved without relying on the study of the ample and
movable cones of $\M(\vv)$.
On the other hand, in the
second case our approach introduces new involutions.
The new construction obtained in this paper can be summarized in the
following result, see Theorem 
\ref{mainth} and Corollary \ref{anticor}. 

\begin{mainthm}\label{MainIntro}
Let $k$ and $g_0 \ge k(k+1)$ be integers. 
  Let $(X,H)$ be a polarized K3 surface such that \correct{$\Pic(X)=\ZZ \cdot H\oplus^{\bot} N$},
  with $H$ ample and $N$ not containing effective
  divisors. Let $\cS$ be a spherical bundle with Mukai vector
  $\vv(\cS)=(2,H,2g_0+1)$ and fix $\vv=(v_0,v_1,v_2)$ with
  $$v_0=2k+1, \qquad v_1=(k+1)H, \qquad v_2 =
  g_0(2k+3)+k+1.$$
  Then $\Phi_{\cS,1}$ defines an anti-symplectic birational involution on the moduli space of sheaves $\M(\vv)$.
\end{mainthm}

Here, $(X,H)$ is a polarized K3 surface of genus $4g_0+2$ and the
condition $g_0 \ge k(k+1)$ amounts to $\vv^2 \ge 0$.
The proof that the involution is well-defined on $\M(\vv)$, which is
to say, that $\Phi_{\cS,1}(\cE)$ is stable for a sufficiently general
sheaf $\cE$ in $\M(\vv)$, is the
point that requires more effort. To achieve this, the first obstacle
is to show that,
for a generic element $\cE$ of $\M(\vv)$ one has $\Ext^1(\cS,\cE)=0$. 
This holds in the above range \correct{of values of $g_0$ and $k$}, essentially owing to results of Yoshioka,
  \cite{yoshioka:crelle}.
Once this is done, we show that applying \correct{the functor $\Phi_{\cS,1}$} indeed gives a
sheaf with the required invariants
and check that this sheaf
is stable. This forces the above choice of $v_2$, which agrees with
the requirement that $\vv=(v_0,v_1,v_2)$ is indeed a fixed vector of the cohomological reflection associated with our
construction \correct{(to be more precise, the condition on $v_2$ is equivalent to the condition $\vv(\Phi_{\cS,1}(\cE))=\vv$ for $\cE\in \M(\vv)$, see Lemmas \ref{v} and \ref{rk1v})}.
Finally, to check that the involution is anti-symplectic, we study the action on the cohomology, see
Proposition \ref{action} for a more detailed statement.

We note that, for $k=0$, our theorem provides involutions on the
Hilbert scheme of points $X^{[n]}$, with $n=(g+2)/4$ for all $g \equiv
2$ modulo $4$, generalizing the construction of \cite{beri2022birational}.

\correct{
  It might even be that any birational involution on \correct{$\M(\vv)$} should be 
induced by an endofunctor of $\Db(X)$ obtained as composition of
spherical twists, tensor product by line bundles, duality and
shifts, together with pull-backs of involutions of $X$ itself.
However, at this stage this speculation is based only on our knowledge of the examples of
birational involutions constructed so far.}

In any concrete case, an obstacle to check that \correct{the functor $\Phi_{\cS,d}$} induces a
birational (respectively, biregular) involution is to check that  
$\Ext^1(\cS,\cE)$ vanishes for
sufficiently general (respectively, for all) sheaves $\cE$ in
$\M(\vv)$, for suitable choices of $\vv$.
This interpolation problem, tightly connected to Brill-Noether theory,
has been treated for generic sheaves in recent and less recent times,
see for instance
\cite{CosNueYos21, yoshioka:brill-noether-rims, MarkmanMain, yoshioka:crelle}.
\correct{Although not much is
known about the vanishing of $\Ext^1(\cS,\cE)$ for an arbitrary sheaf $\cE\in \M(\vv)$, we prove that $\Ext^1(\cS,\cE)=0$ for a particular restricted class of such $\cE\in \M(\vv)$ in Lemma \ref{no H1}.}

The paper is structured as follows. In Section
\ref{preliminares}, we introduce the main concepts and results
related to the moduli space of sheaves on K3 surfaces,
highlighting the role and some properties of (anti)
auto-equivalences of derived categories. In Section
\ref{constru}, we delve into the auto-equivalence
$\Phi_{\cS,d}$ within $\Db(X)$ and explore the
necessary conditions for obtaining an involution on moduli
spaces of stable sheaves. Section \ref{existence} establishes
the existence of the involutions presented in Theorem
\ref{MainIntro}. Additionally, we obtain the involutions
introduced by Markman and studied by O'Grady. Section \ref{examples} describes
several examples; notably, we \correct{prove} that the Beri--Manivel
involution corresponds to $\Phi_{\cS,1}$ within the
context of its action on the moduli space of stable sheaves
$\M(\vv)$, involving a suitable vector $\vv$ and $g$, see Example
\ref{BeriExample}. 
Finally, the action on cohomology is described in Sections \ref{antisympl} and \ref{cohomology}.

\bigskip

\noindent\textbf{Acknowledgements.} 
We would like to express our gratitude
to G. Mongardi for providing specific comments on an earlier
version of the paper. We express our deep gratitude toward P. Beri, L.
Manivel and K. Yoshioka for many valuable
comments and suggestions.

\section{Preliminaries}\label{preliminares}

Here we give a short account of the material needed to set up
our construction, by recalling some fundamental facts about moduli
spaces of semistable sheaves on a projective K3 surface $X$ and
describing the basic equivalences of the derived category of
$X$ and their action on the cohomology of $X$.

\subsection{Moduli space of sheaves on K3 surfaces}\label{notata}
The moduli space of semistable sheaves on a projective K3
surface $X$      
stands as a fundamental example in the study of irreducible
holomorphic symplectic (IHS)
manifolds.
Let $H$ be an ample divisor on $X$ and set $g$ for the genus
of $H$, namely $H^2=2g-2$, so $(X,H)$ is a polarized K3 surface \correct{of genus $g$}.
Let $\HH^*(X,\Z)=\HH^0(X,\Z)\oplus \HH^2(X,\Z)\oplus
\HH^4(X,\Z)$ endowed with the pairing $\cdot$ \correct{be 
the Mukai lattice associated with $X$}. For a
coherent sheaf $\cE$ on $X$, its Mukai vector
$\vv(\cE)=(v_0,c_1,v_2)$ in  
$\HH^*(X,\Z)$ satisfies
$$v_0=\rk(\cE), \qquad c_1=c_1(\cE), \qquad v_2=\chi(\cE) -
\rk(\cE).
$$
This extends to any objects of the derived category $\Db(X)$.
 Applying Riemann-Roch we obtain
$v_2=\rk(\cE)+c_1(\cE)^2/2-c_2(\cE)$. 

The degree of $\cE$, denoted by $\deg_H(\cE)$, is defined as
the normalized intersection of $H$ with the first Chern class of $\cE$,
  i.e., $\deg_H(\cE) \coloneqq \frac {1}{H^2} H\cdot c_1(\cE)$. The slope of $\cE$,
denoted by $\mu_H(\cE)$, is defined
as $$\mu_H(\cE)\coloneqq\frac{\deg_H(\cE)}{\rk(\cE)}, \text{ if } \rk(\cE)
>0.$$ 

The moduli space $\M_H(\vv)$ consists of $H$-semistable
(torsion-free) sheaves on $X$ with Mukai vector $\vv$. \correct{We will omit $H$ from the notation and simply write $\M(\vv)$, since
we will always work with the polarized surface $(X,H)$, the reference
to $H$ being implicit.} The
dimension of $\M(\vv)$ is given by $\vv^2+2 =
c_1^2-2v_0v_2+2$.
Thanks to the contributions of various
authors, including Mukai, G\"ottsche-Huybrechts, O'Grady, and
Yoshioka (see \cite[Proposition 5.1, Theorem 8.1]{Yos01}, it
has been established that under specific restrictions on the
Mukai vector $\vv$, the moduli space $\M(\vv)$ is an
IHS manifold of dimension $2n$ with the same deformation type
as a Hilbert scheme $X^{[n]}$ of $n$ points on a K3 surface $X$.

\subsection{(Anti)-equivalences and their action on the Mukai lattice}\label{introaction}
\label{equivalences}

Let us briefly recall some basic
(anti)-auto-equivalences of $\Db(X)$ and review their action
on the cohomology of the surface $X$.
\subsubsection{Duality}
Duality gives an involutive functor $\Db(X) \to
\Db(X)^{\text{op}}$, i. e. a contravariant functor $\Db(X) \to
\Db(X)$. This is defined by
\[
  \cE \mapsto
  \RcHom(\cE,\cO_X),
\]
where the right-hand side is seen as an object of $\Db(X)$, see \cite[Corollary
5.29]{huybrechts:fourier-mukai}.
The induced mapping at the level of Mukai vectors is
the automorphism $D$ of $\HH^*(X,\Z)$ defined
as $$D:(w_0,c_1,w_2)\mapsto (w_0,-c_1,w_2).$$

\subsubsection{Twisting by a line bundle}

Let $\cL$ be a
line bundle on $X$. Tensoring by $\cL$ gives an
auto-equivalence of $\Db(X)$ defined by sending an object
$\cE$ to $\cE \otimes_{\cO_X} \cL$.
The induced action on Mukai vectors
is
$\cdot \ch(\cL): \HH^*(X,\Z)\rightarrow \HH^*(X,\Z)$, 
sending the Mukai vector of a sheaf $\cF$ to the Mukai vector
of $\cF\otimes\cL$. For
$\cL=\cO_X(dH)$,
we have
{$$
{\cdot \ch(\cO_X(dH))}: (w_0,c_1,w_2) \mapsto
(w_0,c_1+dw_0H,w_2+d^2(g-1)w_0+dc_1\cdot H).$$

\subsubsection{Shifts}

The shift functor on $\Db(X)$ sending
an object $\cE$ to $\cE[1]$ has an induced 
action on cohomology assigning to a Mukai vector $\vv$ its opposite $-\vv$.

\subsubsection{Spherical twists}

Let $\cS$ be an
object of $\Db(X)$. We define $\bT_{\cS}$ as the left mutation
endofunctor on $\Db(X)$ with respect to $\cS$. This
functor applied to an object $\cE$, is the
cone of the natural evaluation morphism
$\be_{\cS,\cE}$, hence we have a distinguished triangle
\begin{equation}
  \label{TS}
  \RHom(\cS,\cE) \otimes
  \cS \xrightarrow{\be_{\cS,\cE}} \cE \to \bT_\cS(\cE)  
\end{equation}

Here, if $\cS$ is not concentrated in a single degree, the
tensor product should be derived as well, however
in this paper $\cS$ will always be a vector bundle.
If $\cS$ is spherical in the sense of \cite[Chapter
16]{huybrechts:K3}, then $\bT_{\cS}$ is an
equivalence, known as a spherical twist (we refer to
\cite[Exercise 8.5 and Proposition
8.6]{huybrechts:fourier-mukai}). The induced action on
$\HH^*(X,\ZZ)$ is given by the reflection along the
hyperplane orthogonal to the $(-2)$-class given by the Mukai
vector $\vv(\cS)$ on $\HH^*(X,\Z)$, see \cite[Lemma
8.12]{huybrechts:fourier-mukai}. 
More explicitly, 
we set $R_{\cS}:\HH^*(X,\Z)\rightarrow \HH^*(X,\Z):x\mapsto x+\left(x\cdot\vv(\cS)\right) \vv(\cS)$. 
Then, according to \cite[Exercise 8.5 and Lemma 8.12]{huybrechts:fourier-mukai}, we have:
$$\vv(\bT_{\cS}(\cE))=-R_{\cS}(\vv(\cE)).$$

\section{Construction of the involution}\label{constru}

Let $X$ be a smooth projective K3 surface polarized by
$H$  and let $\cS$ be a
spherical object of \correct{$D^b(X)$} and $d \in \ZZ$. Here we define an 
anti-auto-equivalence $\Phi_{\cS,d}$ of $\Db(X)$ and show that, under suitable
conditions on $\Pic(X)$ and on the Mukai vectors of $\cS$ and $\vv$, 
the functor $\Phi_{\cS,d}$ defines a birational or even
biregular involution of the moduli space $\M(\vv)$ of
$H$-semistable sheaves on $X$ with Mukai vector $\vv$.

\subsection{General construction as an equivalence of derived categories}\label{31construction}

\correct{The functor $\Phi_{\cS,\cL}^p$ is given as the}
composition of endofunctors of $\Db(X)$ described in Section \ref{equivalences}.
\begin{defi}\label{maindefi}
  Let $\cL \in \Pic(X)$. Set  $\Phi_{\cS,\cL}^p$ for the
  contravariant endofunctor of $\Db(X)$:
  \[
    \Phi_{\cS,\cL}^p:            \cE \mapsto \Phi_{\cS,\cL}(\cE) = \RcHom(\bT_{\cS}(\cE),\cL[p]).
  \]
  \correct{We omit $p$ when $p=1$}. 
For fixed $H \in \Pic(X)$, given $d \in \ZZ$ we write $ \Phi_{\cS,d}=
  \Phi_{\cS,\cO_X(dH)}^1$, so
  \[
    \Phi_{\cS,d}(\cE) = \RcHom(\bT_{\cS}(\cE),\cO_X(dH)[1]).
  \]
For simplicity, we often abbreviate
$\Phi_{\cS,\cL}^p$ to $\Phi_{\cS,\cL}$ or just $\Phi$.
\end{defi}

The functor $\Phi_{\cS,\cL}^p$ is an anti-auto-equivalence 
of $\Db(X)$ as \correct{it is a}
composition of equivalences and duality (see Section
\ref{equivalences}).
\correct{We now show that :}
\begin{lemma} \label{involutive functor}
  If $\cS \simeq \RcHom(\cS, \cL[q])$ for some $q \in \ZZ$, then
  $\Phi=\Phi_{\cS,\cL}^p$ 
  is involutive for all $p \in \ZZ$.
  If $\Pic(X) \simeq \ZZ\cdot H$ and $\cS$ a torsion-free sheaf of
  rank $s_0$, this
  happens if and only if 
  \begin{equation} \label{rank = 1 or 2}
    s_0 \in \{1,2\} \qquad \mbox{and} \qquad    \cL \simeq \det(\cS)^{\frac 2{s_0}}.
  \end{equation}
\end{lemma}
\begin{proof}
  Recall from \cite[Definition 2.7]{seidel-thomas} the notion of dual spherical twist $\bT_\cS'$, sending an object $\cE$ of $\Db(X)$
  to the cone of the dual evaluation map:
  \[
    \cE \to \RHom(\cE,\cS)^\vee \otimes \cS.
  \]
  According to 
  \cite[Proposition 2.10]{seidel-thomas}, $\bT_\cS \circ \bT_\cS'$ and
  $\bT_\cS' \circ \bT_\cS $ have natural transformations to the
  identity functor.
  Under the condition $\cS \simeq \RcHom(\cS, \cL[q])$, for any object $\cE$ of $\Db(X)$ we have a natural isomorphism:
  \[
    \bT_\cS'(\RcHom(\cE,\cL)) \simeq \RcHom(\bT_\cS(\cE),\cL).
  \]
This implies that $\Phi (\Phi(\cE))$ is naturally isomorphic to $\cE$.
  
	If $\cS$ is locally free and \eqref{rank = 1 or 2} holds, then $\cS
  \simeq \cS^\vee \otimes \cL \simeq \RcHom(\cS, \cL)$ so that $\Phi$
  is involutive.
  Conversely, assume that $\cS$ is a torsion-free sheaf,
  suppose $\Pic(X) \simeq H \cdot \ZZ$ for some ample divisor $H$,
  with $H^2=2g-2 \ge 2$. Write $\s=(s_0,s_1H,s_2)$ for the Mukai
  vector of $\cS$.
  The assumption $\cS \simeq \RcHom(\cS, \cL[q])$ yields $q=0$
  and we get that $\cS$ is locally free.
  Moreover, writing $c_1(\cS)=s_1H$ and $c_1(\cL)=d
  H$, we get $s_0d = 2s_1$.
  Then, the condition $\s^2=-2$ reads
  \[
    s_0^2d^2(g-1)=4(s_0s_2-1).
  \]

  For any prime divisor $p > 2$ of $s_0$, the right-hand-side is
  non-zero modulo $p$, which is a contradiction.
  Hence $s_0=2^q$ for some $q \ge 0$ and again
  the right-hand-side is non-zero modulo $2^q$ if $q > 2$. Similarly
  one excludes the case $q=2$ and we are left with $s_0 \in \{1,2\}$.
  Now \eqref{rank = 1 or 2} follows from $s_0d = 2s_1$.
\end{proof}

The image of a semistable sheaf by $\Phi$ 
may not be a coherent sheaf, and even then, it is not necessarily
semistable.
Here we provide sufficient conditions for $\Phi(\cE)$ to be a coherent sheaf.

\begin{lemma}\label{mainlemma}
  \correct{Set $p=1$, choose a spherical sheaf $\cS$ and a line bundle $\cL$ on $X$. Assume that:
  \begin{enumerate}[label=\roman*)]
  \item \label{main-i} the sheaves
    $\cE$ and $\cS$ are torsion-free;
  \item \label{main-ii} we have the vanishing condition \label{vanishing 1 et 2}
    $\Ext^1(\cS,\cE)=\Ext^2(\cS,\cE)=0$;
  \item \label{main-iii} the sheaf 
    $\coker\left(\be_{\cS,\cE}:\Hom(\cS,\cE) \otimes \cS \to \cE\right)$ is a torsion sheaf.
  \end{enumerate}}
  Then $\Phi(\cE)=\Phi_{\cS,\cL}^1(\cE)$ is a sheaf that fits in the following exact sequence:
  \begin{equation}\label{eq: the map phi}
    0 \rightarrow   \cE^\vee\otimes \cL \rightarrow
    \Hom(\cS,\cE)^\vee \otimes \cS^\vee \otimes \cL \rightarrow  \Phi_{\cS,\cL}(\cE) \rightarrow \cExt^1(\cE,\cL) \rightarrow 0.
  \end{equation}
\end{lemma}
\begin{proof}
  Recall that $\cS$ is locally free.
  Then, applying $\RcHom(-,\cL)$ to the distinguished
  triangle \eqref{TS},  we obtain the distinguished triangle:
  \begin{equation*} 
    \RcHom(\cE,\cL) \rightarrow \RHom(\cS,\cE)^\vee
    \otimes \cS^\vee \otimes \cL \rightarrow \Phi(\cE).
  \end{equation*}
  
  Taking homology leads to a long exact sequence:
  \begin{equation}\label{simple}
    \cdots \to \cExt^k(\cE,\cL) \to \Ext^k(\cS,\cE)
    \otimes \cS^\vee \otimes \cL \to \cH^k(\Phi(\cE)) \to
    \cExt^{k+1}(\cE,\cL) \to \cdots
  \end{equation}
  Since $\coker(\be_{\cS,\cE})$ is a torsion sheaf, the
  transpose of $\be_{\cS,\cE}$ is an injective map
  \[
    \cHom(\cE,\cL) \to \Hom(\cS,\cE) \otimes
    \cS^\vee \otimes \cL.
  \]
  Since the kernel of this map is $\cH^{-1}(\Phi(\cE))$,
  we get $\cH^{-1}(\Phi(\cE))=0$.
  Moreover since $\cE$ is torsion-free, $\cExt^p(\cE,\cL)=0$ for all $p\geq 2$. 
  Hence, under the assumption \ref{vanishing 1 et 2},
  $\cH^p(\Phi(\cE))$ vanishes for all $p\in\Z^*$ and
  $\Phi(\cE) \simeq \cH^0(\Phi(\cE))$. 
  In other words, $\Phi(\cE)$ is a coherent sheaf and
  \eqref{simple} for $k=0$ provides the desired exact sequence.
\end{proof}

\begin{lemma} \label{d} \correct{Under the assumptions of Lemma
  \ref{mainlemma}}, let $\cN \in \Pic(X)$. Set:
  \[\cL'=\cL\otimes \cN^{\otimes 2}, \qquad \cS'=\cS\otimes
    \cN, \qquad            \vv'=\vv \cdot \ch(\cN)
  \]
  and define $f:\M(\vv) \to \M(\vv')$ by
  sending $\cE$ to $\cE'=\cE \otimes \cN$.
  \correct{If $\Phi_{\cS,\cL}(\cE) \in \M(\vv)$, then
  $\Phi_{\cS',\cL'}(\cE') \in \M(\vv')$ and:
  $$\Phi_{\cS',\cL'}(\cE')=f \circ \Phi_{\cS,\cL}\circ f^{-1}(\cE').$$}
\end{lemma}
\begin{proof}
  \correct{Let us first work regardless of the assumption $p=1$ and
    of the conditions \ref{main-i}, \ref{main-ii} and \ref{main-iii} of Lemma
  \ref{mainlemma}. For any object $\cE$ of $\Db(X)$ we have
  \[
    \bT_{\cS'}(\cE \otimes \cN) \simeq   \bT_\cS(\cE) \otimes\cN.
  \]
  Therefore, for all $p \in \ZZ$, we get:
  \begin{align*}
    \Phi_{\cS',\cL'}^p(\cE \otimes \cN)&= \RcHom(\bT_{\cS'}(\cE \otimes \cN), \cL'[p])
                                         \simeq \\
    & \simeq
    \RcHom(\bT_{\cS}(\cE), \cL \otimes \cN[p]) = \Phi_{\cS,\cL}^p(\cE) \otimes \cN.
  \end{align*}
  Now, for $\cE \in \M(\vv)$, if $p=1$ and the conditions \ref{main-i}, \ref{main-ii} and \ref{main-iii} of Lemma  \ref{mainlemma} hold, then the
  object $\Phi_{\cS,\cL}^1(\cE)$ is a coherent sheaf.
  It this sheaf lies in $\M(\vv)$ then $\Phi_{\cS,\cL}(\cE)\otimes \cN \simeq \Phi_{\cS',\cL'}(\cE')$ is a
  sheaf belonging to $\M(\vv')$ and the desired formula holds.
}
\end{proof}
\begin{conse} \label{the consequence}
  As Lemma \ref{d} shows, while dealing with $\cL =
  \cO_X(dH)$, modulo conjugation, we can assume that $d\in\left\{0,1\right\}$.
  So without loss of generality, when
  $\Pic(X) \simeq \ZZ\cdot H$, we assume that $d\in\left\{0,1\right\}$
  for the \correct{rest} of 
  this paper.
  Note that, in view of Lemma \ref{involutive functor}, we deduce:
  \begin{itemize}
  \item either $\cS \simeq \cO_X$, $\s=(1,0,1)$ and $d=0$,
  \item or $\cS$ is a spherical stable bundle of rank $2$, the genus
    $g$ is even, $\s=(2,H,\frac g 2)$ and $d=1$. 
  \end{itemize}
\end{conse}

\subsection{Strategy to obtain an involution on \correct{$\M(\vv)$}} 

From now, we assume that $\cE$ and $\cS$ are torsion-free sheaves. 
We set $\vv(\cE)=(v_0,v_1H,v_2)$ and $\s=(s_0,s_1H,s_2)=v(\cS)$ the Mukai \correct{vectors} of $\cE$ and $\cS$ respectively.
\begin{obj}\label{mainobj}
  The objective of the paper is to determine when $\Phi$ defines an automorphism \correct{of} the moduli space $\M(\vv)$. Therefore, there are two main properties that $\Phi(\cE)$ has to \correct{satisfy}:
  \begin{enumerate}[label=(\roman*)]
  \item \label{oi}
    $\vv(\Phi(\cE))=\vv$.
  \item \label{oii}
    $\Phi(\cE)$ is torsion-free and semi-stable.
  \end{enumerate}
\end{obj}

\subsubsection{Determination of Mukai vectors}

In this subsection, we work under the assumption that $\Pic(X)=\ZZ\cdot H $, where $H$ is
an ample divisor with $H^2=2g-2$, even though later on we will
\correct{allow} a slightly more general Picard lattice.
The reason \correct{for} looking at $\Pic(X) \simeq \ZZ$ is that we would like to
study involutions that do exist for a moduli space $\M(\vv)$ over $X$,
with $X$ general.
By the description of the action on cohomology of the functors
defining $\Phi$ given in \S \ref{equivalences}, we get:

\begin{lemma}\label{v} The action of $\Phi$ on Mukai vectors
  satisfies:
  \begin{equation}
    \vv(\Phi(\cE))=-\left(D\circ
      R_{\cS}(\vv(\cE))\right)\cdot \ch(\cO_X(dH)).
    \label{mukaivectorcondition}
  \end{equation}
\end{lemma}

\correct{From Lemma \ref{v}}, we can determine the different
possibilities \correct{for} $\cS$ with Mukai vector $\s$ to verify condition \ref{oi} of
Objective \ref{mainobj} when $d=0$ and when $d=1$. 

\begin{lemma}\label{rk1v}
  Let $\vv$ be a Mukai vector with $
    \s\cdot\vv\neq0$.
  Then $\vv(\Phi(\cE))=\vv$ if and only if:
  \begin{enumerate}[label=\roman*)]
  \item \label{s0=1} $(s_0,d)=(1,0)$, $\cS = \cO_X$ and $v_0=v_2$, or:
  \item \label{s0=2} $(s_0,d)=(2,1)$, $\cS$ is spherical with
    $\s=(2,H,\frac{g}{2})$ and $2v_2=(2g-2)v_1-v_0(\frac{g}{2}-1)$.
  \end{enumerate}
\end{lemma}

\begin{proof}
  In case \ref{s0=1}, we know that the spherical bundle is $\cO_X$ and
  that $d=0$,
  see Consequence \ref{the consequence}. So $\s=(1,0,1)$.
  Moreover, looking at the rank of the involved sheaves, Lemma \ref{v} gives
  $v_0=-\s\cdot\vv-v_0$, hence
  $v_0=v_2$.
  To finish the proof, it only remains to check that
  $\chi(\Phi(\cE))=\chi(\cE)$ when $\s=(1,0,1)$ and $v_0=v_2$.
  However this follows from the equation on the Euler characteristics obtained from (\ref{mukaivectorcondition}).

    Let us look at case \ref{s0=2}. Set $a\coloneqq -\s\cdot\vv$.
  Lemma \ref{v} provides equations on the ranks and the first Chern
  classes, to the effect that 
  $2v_0=a s_0$ and $v_1-v_0=-as_1+v_1.$
  We get 
  \begin{equation}
    v_0=as_1.
    \label{aaa}
  \end{equation}
  Recall from Consequence \ref{the consequence} that 
  $\s=(2,H,\frac g2)$, so that \eqref{aaa} yields $v_0=a$. We obtain
  \begin{equation}
    2v_2=(2g-2)v_1-v_0\left(\frac{g}{2}-1\right).
    \label{v0v2}
  \end{equation}
  It remains to verify that we have 
  $\chi(\Phi(\cE))=\chi(\cE)$.
  Lemma \ref{v} provides the following equation on the Euler characteristics:
  $$\chi((\Phi(\cE)\otimes\cO_X(-H))=v_0\left(\frac{g}{2}+2\right)-v_2-v_0.$$
  However:
  $$\chi(\cE(-H))=v_2+v_0-(g-1)(2v_1-v_0).$$
  So, we need to verify that:
  $$v_2+v_0-(g-1)(2v_1-v_0)=v_0\left(\frac{g}{2}+2\right)-v_2-v_0.$$
  This is equivalent to relation \eqref{v0v2}.
\end{proof}

\subsubsection{Preservation of stability}

With the Mukai vectors determined in the previous subsection,
we now prove that the functor \correct{$\Phi$} preserves
\correct{semistability of torsionfree sheaves}.
In this subsection, we make a slightly more general assumption on
  the Picard group of $X$; we assume that \correct{$\Pic(X)= \ZZ\cdot H\oplus^{\bot}
  N$}, where $H$ is an ample divisor with $H^2=2g-2$ and $N$ is a
  lattice that does not contain any effective divisor. 

\begin{prop}\label{torsionfree} 
  We assume that $\Ext^1(\cS,\cE)=\Ext^2(\cS,\cE)=0$. \correct{As before, we denote by $\be_{\cS,\cE}$ the evaluation map (see (\ref{TS})).}
  \begin{enumerate}[label=(\roman*)]
  \item \label{pi}
    If $\dim(\coker(\be_{\cS,\cE})) \le 0$, then
    $\Phi(\cE)$ is a torsion-free sheaf.
  \item \label{pii}
    If in addition $\Ker(\be_{\cS,\cE})$ is slope-stable, then $\Phi(\cE)$ is slope-stable.
  \end{enumerate}
\end{prop}
\begin{proof}
  Set $\cK\coloneqq \Ker(\be_{\cS,\cE})$ and $\cE_0\coloneqq \im(\be_{\cS,\cE})$.
  We obtain an exact sequence:
  \begin{equation} 			\label{kerIm}
    0 \to \cK \to \Hom(\cS,\cE)\otimes\cS \to
    \cE_0\to 0.
  \end{equation}
  Since $\coker(\be_{\cS,\cE})$ has dimension at most 0, $c_1(\cE_0)=c_1(\cE)$ and $\cE^\vee=\cE_0^\vee$. 
  We can deduce the first Chern class of $\cK$:
  \begin{equation}
    c_1(\cK)=-\left(v(\cS)\cdot v\right) c_1(\cS)-c_1(\cE).
    \label{c1K}
  \end{equation}
  Note also that $\cK$ is a reflexive sheaf, hence a vector bundle.
  
  By Lemma \ref{mainlemma}, $\Phi(\cE)$ is a sheaf and \correct{satisfies} (\ref{eq: the map phi}).
  Consider the
  image 
  $\cG$ of the middle map in the exact sequence.
  We have:
  \begin{equation}
    0 \to \cG \to \Phi(\cE) \to \cExt^1(\cE,\cO_X(dH)) \to 0.
    \label{middle0}
  \end{equation}
  Note that $\cExt^1(\cE,\cO_X(dH))$ is supported on isolated points since $\cE$ is torsion free.
  Taking the dual, tensored by $\cO_X(dH)$, of
  \eqref{kerIm} we see that $\cG$ is also the image of the middle map of the
  following sequence:
  \begin{equation}
    0 \to \cE_0^\vee(dH) \to
    \Hom(\cS,\cE)^\vee\otimes\cS^\vee(dH)
    \to \cK^\vee (dH) \to
    \cExt^1(\cE_0,\cO_X(dH))\to 0.
    \label{middle3}
  \end{equation}
  In particular, we have $\cG^{\vee}\simeq
  \Phi(\cE)^{\vee}$ and $\cG^{\vee} \simeq \cK(-dH)$.
  This shows that $c_1(\cG^{\vee})=c_1(\Phi(\cE)^{\vee})$ and $c_1(\cG^{\vee})=c_1(\cK(-dH))$.
  Combined with (\ref{c1K}) and (\ref{mukaivectorcondition}), we obtain that:
  $$c_1(\Phi(\cE)^{\vee})=-c_1(\Phi(\cE)).$$
  This proves that $\Phi(\cE)$ cannot have 1-dimensional torsion.
  
  Moreover, we know that $\Phi$ is an anti-auto-equivalence of $\Db(X)$, which
  implies
  \[
    \Ext^i(\Phi(\cE),\Phi(\cE)) \simeq
    \Ext^i(\cE,\cE), \qquad \forall i \in \ZZ.
  \]
  Therefore, since $\cE$ is stable and hence simple, $\Phi(\cE)$
  is also simple.
  This implies that $\Phi(\cE)$ has no zero-dimensional torsion. To see
  this, note that if
  $p \in Z$ was a point in the support of a zero-dimensional torsion
  subsheaf of $\Phi(\cE)$, since $\Phi(\cE)$ has positive rank, one could find
  an epimorphism of $\Phi(\cE)$ onto the skyscraper sheaf $\cO_p$, and 
  since $\cO_p$ maps injectively into $\Phi(\cE)$, we would get an
  endomorphism of $\Phi(\cE)$ factoring through $\cO_p$,
  which is absurd. So \ref{pi} is proved.
  
  We are ready to prove \ref{pii}. Since $\cK$ is a
  vector bundle. The slope-stability of $\cK$ implies
  the slope-stability of $\cK^\vee(dH)$. However by
  (\ref{middle3}), $\cG$ differs from $\cK^\vee(dH)$
  only along a zero-dimensional subset of $X$. So $\cG$
  is slope-stable. Similarly, by (\ref{middle0}) $\cG$
  and $\Phi(\cE)$ differs along a zero-dimensional
  subset of $X$. We obtain that $\Phi(\cE)$ is slope
  stable. 
\end{proof}

When dealing with twisting about $\cO_X$, we are able to prove
  Objective \ref{mainobj}, \ref{oii} when $v_1=1$,
so the next step would be to
establish the following lemma. We omit its proof since the ideas
  are similar to the ones used by Markman
  (see \cite[P. 682--684]{MarkmanMain} and \cite[Lemma 4.10]{OGrady})
  and the proof is also similar to the one of Proposition
  \ref{stability2} \correct{below}.

\begin{lemma}\label{stability1}
  We assume that:
  \begin{itemize}
  \item[(i)]
    $\Ext^1(\mathcal{O}_X,\cE)=\Ext^2(\mathcal{O}_X,\cE)=0$ and  $\Hom(\mathcal{O}_X,\cE)\neq0$; 
  \item[(ii)]
    $\vv(\Phi(\cE))=\vv(\cE)$ and $c_1(\cE)=H$;
  \item[(iii)]
    \correct{$\Pic(X)=\Z \cdot H \oplus^{\bot} N$} with $N$ not containing any effective divisor.
  \end{itemize}
  Then $\coker(\be_{\mathcal{O}_X,\cE})$ is supported on isolated points and $\Ker(\be_{\mathcal{O}_X,\cE})$ is slope-stable.
\end{lemma}

When dealing with involutions defined by twisting about a spherical
bundle $\cS$ of rank $2$, our choice for the Mukai vector is made so
that the $H$-slope of $\cE$ is just a bit bigger than $1/2$.
Then, in order to guarantee preservation of stability, we need the
following proposition.

\begin{prop}\label{stability2}
  Let $\cS$ be a spherical bundle of Mukai vector $(2,H,\frac{g}{2})$ with $g$ even.
  Assume:
  \begin{enumerate}[label=(\roman*)]
  \item
    $\Ext^1(\cS,\cE)=\Ext^2(\cS,\cE)=0$ and  $\Hom(\cS,\cE)\neq0$; 
  \item
    $\rk(\cE)=2k+1$ with $k\in\N$;
  \item
    $\vv(\Phi(\cE))=\vv(\cE)$ and $c_1(\cE)=(k+1)H$;
  \item
    \correct{$\Pic(X)=\Z \cdot H \oplus^{\bot} N$} with $N$ not containing any effective divisor.
  \end{enumerate}
  Then $\coker(\be_{\cS,\cE})$ is supported on isolated points and $\Ker(\be_{\cS,\cE})$ is slope-stable.
\end{prop}
\begin{proof}
  We first prove that $\coker(\be_{\cS,\cE})$ is supported on isolated points.
  We denote by
  $\cE_0$ the image of $\be_{\cS,\cE}$ and by $v_0$ the rank of $\cE$.
  Let $c \in \ZZ$ and $D\in N$ such that $c_1(\cE_0)=c H+D$. By stability of $\cE$ and $\cS$,
  we must have:
  $$\frac{1}{2}\leq \frac{c}{\rk(\cE_0)}\leq \frac{k+1}{2k+1}.$$
  First, we notice that we may not have $\frac{1}{2}< \frac{c}{\rk(\cE_0)}< \frac{k+1}{2k+1}$.
  Indeed, this would imply
  $$\frac{\rk(\cE_0)}{2}< c<\frac{\rk(\cE_0)}{2}+ \frac{\rk(\cE_0)}{4k+2}<\frac{\rk(\cE_0)}{2}+\frac{1}{2}.$$
  But this is impossible for any choice of the integers $c$ and $\rk(\cE_0)$.
  \medskip
  
  Now, we consider the case $\frac{c}{\rk(\cE_0)}=\frac{1}{2}$. We write $\cK=\Ker(\be_{\cS,\cE})$.
  We have $\rk(\cE_0)=2c$ and $\rk(\cK)=2\rk(\cE)-\rk(\cE_0)=2(\rk (\cE)-c)$.
  Moreover $c_1(\cK)=(\rk(\cE)-c)H-D$. Hence the slope of $\cK$ is also $\frac{1}{2}$.
  Therefore by stability of $\cS$, we must have $\cK \simeq \cS^{\oplus (\rk(\cE)-c)}$. However, this is impossible because $\Hom(\cS,\cK)=0$.
  Indeed, we have the exact sequence:
  $$0 \to \cK \to \Hom(\cS,\cE)\otimes\cS \to \cE;$$
  if we take the image by the functor $\Hom(\cS,-)$, since $\cS$ is simple we obtain:
  $$0 \to \Hom(\cS,\cK) \to \Hom(\cS,\cE) \to \Hom(\cS,\cE).$$

  Note that the last map is induced by the evaluation map and therefore
  it is just the identity. We obtain $\Hom(\cS,\cK)=0$.
  \medskip
  
  The only remaining case is  $\frac{c}{\rk(\cE_0)}=\frac{k+1}{2k+1}$. Since $k+1$ and $2k+1$ are coprime and
  $\rk(\cE_0)\leq 2k+1$, we obtain $\rk(\cE_0)=2k+1=\rk(\cE)$ which is
  what we wanted to prove.
  We also get $c_1(\cE_0)=c_1(\cE)$, hence $c_1(\cE/\cE_0)=0$ so that
  $\coker(\be_{\cS,\cE}) \simeq \cE/\cE_0$ is supported on isolated points.  
  
  \medskip
  
  Now, we prove that $\cK$ is slope-stable.
  We assume that there exists a destabilizing subsheaf $\cH\hookrightarrow \cK$ and we will find a contradiction.
  We have $\cH\hookrightarrow \cK\hookrightarrow \Hom(\cS,\cE)\otimes\cS$. The slope of $\cK$ is given by $\frac{\rk(\cE)-c}{\rk(\cE)}=\frac{k}{2k+1}$.
  Hence by stability of $\cS$, we obtain:
  $$\frac{k}{2k+1}\leq \frac{\alpha}{\rk(\cH)}\leq \frac{1}{2},$$
  with $c_1(\cH)=\alpha H+D'$ with $D'\in N$.
  Note that $\rk(\cH)< 2k+1$. 
  As before, we have:
  $$\frac{\rk(\cH)}{2}-\frac{1}{2}<\frac{\rk(\cH)}{2}-\frac{\rk(\cH)}{4k+2}\leq \alpha\leq \frac{\rk(\cH)}{2}.$$

  The only solution is then $\rk(\cH)=2\alpha$. As before this implies
  that $\cH=\cS^{\oplus\alpha}$ and it is impossible because we have
  seen that $\Hom(\cS,\cK)=0$. Therefore $\cK$ is slope-stable, which
  finishes the proof. 
\end{proof}

\section{Existence of the involution}\label{existence}

Here we describe the argument to see that the functor \correct{$\Phi_{\cS,d}$}
indeed induces an involution on the moduli space $\M(\vv)$ over $X$.
We first deal with the case $\cS=\cO_X$, where we recover
results of Markman and O'Grady, then move to the case
$\rk(\cS)=2$, where we describe new involutions.

\subsection{Markman-O'Grady's involutions}\label{MarkmanInv}

In this section, we review Markman's results \cite{MarkmanMain} (see also \cite[Section 4]{OGrady}).
This corresponds to the spherical twists about $\cS=\cO_X$, with
$d=0$, see Consequence \ref{the consequence}.
In this section, for simplicity, we denote $\Phi_{\cO_X,0}=\Phi$.

\correct{Our approach provides a more direct path to check regularity of
Markman-O'Grady's involutions.
Indeed, in \cite{MarkmanMain} Markman constructs birational maps between moduli spaces
$\M(\vv) \dashrightarrow \M(\vv')$
for appropriate pairs of Mukai vectors $\vv$ and $\vv'$.
Involutions arise when $\vv=\vv'$, which in turn forces $\vv=(r,H,r)$.
O' Grady analyzed these involutions again in \cite{OGrady} and showed their
antisymplectic nature.
These are the
involutions described in our paper on the space $\M_X(r,H,r)$ as
induced by $\Phi=\Phi_{\cO_X,0}^1$,
when $X$ is a K3 surface with Picard number one and genus $g$ and $r \in \NN$ satisfies $r^2 \le g < (r+1)^2$.
	}


\correct{Markman studies the base loci of the rational maps $\M(\vv)  \dashrightarrow  \M(\vv')$ in general.
  In our case, we propose a sheaf-theoretic proof that
  Markman-O’Grady’s involution is regular by showing directly that
$\HH^1(\cE)=0$ for all $\cE \in \M(\vv)$,
which then allows us to apply Proposition \ref{torsionfree} and Lemma
\ref{stability1}. This replaces Markman's study of the base loci of
the rational maps $\M(\vv) \dashrightarrow \M(\vv')$ and appears to us as a shortcut with respect to Markman's approach.}

\begin{lemma}  \label{no H1}
  Assume $r^2 \le g < (r+1)^2$. 
  Let $\cE$ be any element of $\M(r,H,r)$. Then $\HH^1(\cE)=0$.
\end{lemma}

\begin{proof}
  
  Assume $\HH^1(\cE) \ne 0$ and we will find a contradiction. Note that, by Serre duality, we have
  $\Ext^1(\cE,\cO_X)\simeq \HH^1(\cE)^\vee \ne 0$. Then, we may choose a
  non-zero element of $\Ext^1(\cE,\cO_X)$ and write the corresponding
  non-trivial extension, which takes the form:
  \begin{equation}
    \label{exte}
    0 \to \cO_X \to \cF \to \cE \to 0.  
  \end{equation}

  \correct{Observe that, since the
  bound on $g$ is chosen so that $\vv(\cF)^2 < -2$, the sheaf $\cF$ cannot be slope-stable.
  Then, there is saturated
  destabilizing subsheaf $\cK$ of $\cF$. We set $\cQ= \cF/\cK$ so that
  $\cQ$ is a non-trivial torsion-free coherent sheaf on $X$}. Put $\cK'$ for the
  intersection of $\cK$ and $\cO_X$ in $\cF$, $\cK'' = \cK/\cK'$,
  $\cQ'=\cO_X/\cK'$ and $\cQ''=\cE/\cK''$.

Let us summarize the notation in the following commutative exact diagram:
\footnotesize
\[
\xymatrix@-1ex{
& 0 \ar[d] & 0 \ar[d] & 0 \ar[d]\\
0 \ar[r] & \cK' \ar[r] \ar[d] & \cK \ar[r] \ar[d] & \cK'' \ar[r] \ar[d]& 0\\
0 \ar[r] &  \cO_X \ar[r] \ar[d]&  \cF \ar[r] \ar[d]&  \cE \ar[r] \ar[d]&  0\\
0 \ar[r] & \cQ' \ar[r] \ar[d]& \cQ \ar[r] \ar[d]& \cQ'' \ar[r] \ar[d]& 0. \\
& 0 & 0 & 0
}
\]
\normalsize
  Set $c = \deg_H(\cK)$. Then, since $\cK$ destabilizes $\cF$, we must have:
  \begin{equation}
    \label{desta}
    c \ge 1.
  \end{equation}
	
  Consider the sheaf $\cK'$ and assume $\cK' \ne 0$. Since $\cO_X$
  is locally free of rank $1$, the sheaf $\cK'$ must be torsion-free of
  rank $1$. Therefore, $\cQ'$ is a torsion sheaf, which cannot be
  non-trivial since $\cQ$, and hence $\cQ'$, are torsion-free
  sheaves. This says that $\cQ \simeq \cQ''$. Therefore
  $\rk(\cK'')<r$ as otherwise $\cQ''$ would be a torsion
  sheaf. By stability of $\cE$, we must then have $c\le
  \rk(\cK'') /r$, which is to say $c \le 0$. But this contradicts \eqref{desta}.
  
  We have thus proved that $\cK'=0$, which in turn gives $\cQ' \simeq
  \cO_X$ and $\cK \simeq \cK''$. Since $\cE$ is stable, whenever
  $\rk(\cK'')<r$, we get $c \le 0$, which again contradicts
  \eqref{desta}. Therefore 
  $\rk(\cK'')=r$, so that $\cQ''$ is a torsion sheaf supported on
  a closed subscheme $Z \subsetneq X$. In case $\dim(Z)=1$, then $Z$ is
  a divisor of class $(1-c)H+D$, with $D\in N$. Since $Z$ is effective we would have $(1-c)>0$ 
  which is against \eqref{desta}.
  
  Therefore $Z$ has finite length. This in turn gives
  $c_1(\cQ)=0$. Now, since the sheaf $\cQ$ is torsion-free of rank $1$ and
  with $c_1(\cQ)=0$, the injection $\cO_X \to \cQ$ must be an
  isomorphism, as one sees by composing with the embedding $\cQ \mono
  \cQ^{**} \simeq \cO_X$.
  In other words, composing the injection $\cO_X \mono \cF$ with the projection $\cF
  \to \cQ \simeq \cO_X$ we get an isomorphism. This says that the
  extension \eqref{exte} \correct{splits. This shows that $\Ext^1(\cE,\cO_X)=0$, a contradiction.}
\end{proof}

We deduce the (well-known) regularity \correct{of} Markman-O'Grady's involutions.

\begin{prop}\label{inv}
  Let $X$ be a K3 surface. \correct{Assume that, with respect to the intersection pairing defined on $\Pic(X)$, we have an orthogonal decomposition $\Pic(X)=\ZZ \cdot H\oplus^{\bot}N$} with
  $H^2=2g-2$ and $N$ a lattice which does not contain any effective divisor. Let $r \ge 1$ be an integer with $r^2 \le g < (r+1)^2$. 
  Then, whenever $\cE$ lies in $\M(r,H,r)$, also $\Phi(\cE)$ lies in
  $\M(r,H,r)$. 
	\correct{Therefore the autoequivalence $\Phi$ induces a biregular involution $\Phi$ of $\M(r,H,r)$.}
\end{prop}
\begin{proof}
  By Lemma \ref{rk1v}, we already know that
  $\vv(\Phi(\cE))=\vv(\cE)$. So, it remains to prove that $\Phi(\cE)$
  is slope-stable and torsion-free. 
  This is a consequence of Proposition \ref{torsionfree} and Lemma
  \ref{stability1} if we know that
  $\Ext^1(\cO_X,\cE)=\Ext^2(\cO_X,\cE)=0$. 
  However, the first vanishing is the content of Lemma \ref{no H1}
  while the second one follows from \correct{the} stability of $\cE$.
\end{proof}

\begin{rmk}
  Note that Proposition \ref{inv} can be generalized when $g \geq
  (r+1)^2$, in this case, we only obtain a birational involution, \correct{as}
  follows for instance from Lemma \ref{stability1} and \cite[Theorem
  1.1]{CosNueYos21} (it is also explained in \cite[Section
  4.2.2]{OGrady}).
  The case $r=1$ corresponds to the Beauville
  involutions (see for instance \cite[Section
  4.1.2]{OGrady} and Example \ref{Beauvilleinv}). 
\end{rmk}

\subsection{Involutions from spherical bundles of rank two}

In this section, we study the involution $\Phi_{\cS,1}$ with
$\vv(\cS)=(2,H,\frac{g}{2})$, see Lemma \ref{rk1v}. The main
point \correct{is} to define new involutions from spherical twists about
bundles of rank $2$, which is the content of our main result, namely
Theorem \ref{MainIntro}.
We are in position to prove \correct{it} here, except for the statement on the
anti-symplectic nature of the involution, which will be checked in the
next section.

\begin{thm}\label{mainth}
  Let $(g,k)\in \NN^2$ with $g\equiv 2$ modulo $4$. 
  Let $(X,H)$ be a polarized K3 surface such that \correct{$\Pic(X)=\ZZ \cdot H\oplus^{\bot} N$} with
  $H^2=2g-2$ and $N$ not containing any effective divisor.
  Let $\cS$ be the spherical bundle \correct{with} Mukai vector $(2,H,\frac{g}{2})$. 
    Consider the Mukai vector $\vv=(v_0,v_1,v_2)$ with
  \[
    v_0=2k+1, \qquad v_1 = (k+1)H, \qquad 2v_2=(2g-2)(k+1)-(2k+1)(\frac{g}{2}-1),\]
  as given by Lemma \ref{rk1v}.
  We assume that $\dim (\M(\vv))\geq 2$, namely $g\geq (2k+1)^2+1$.
  Then $\Phi_{\cS,1}$ induces a birational involution on $\M(\vv)$.
\end{thm}

One of the points we need in order to prove this result is to check that, for
$\cE$ sufficiently general in $\M(\vv)$, we have $\Ext^1(\cS,\cE)=0$.
However, this essentially follows \cite[Lemma 2.6]{yoshioka:crelle}.
We provide here a self-contained proof of this fact.

\begin{lemma}\label{noExt}
  Let $g$ be an integer and let $X$ be a K3
  surface having \correct{$\Pic(X)=\ZZ \cdot H\oplus^{\bot} N$}, with $N$ \correct{not containing} any effective divisor. We consider two Mukai vectors $\s=(s_0,s_1H,s_2)$ and $\vv=(v_0,v_1H,v_2)$ which verify the following conditions:
  \[
    \s^2=-2,
    \qquad \vv\cdot \s\leq 0, \qquad 
    s_1v_0<s_0v_1.
  \]
  Let $\cS$ be the vector bundle of Mukai vector $\s$.
  We assume that $\Ext^1(\cS,\cE) \ne 0$ for $\cE$
  generic in $\M(\vv)$. Then there is a Zariski
  open dense subset $U\subset 
  \M(\vv)$ such that for all $\cE\in U$, a generic nonsplitting
  extension
  $$0 \to \cS \to \cF \to \cE \to 0$$
  gives a non-simple sheaf $\cF$.
\end{lemma}
\begin{proof}
  By assumption, we have
  $\Ext^1(\cS,\cE)^\vee \simeq \Ext^1(\cE,\cS) \neq 0$ for all $\cE$ in $\M(\vv)$.
  We define
  \[a = \min\{h^1(\cS ^\vee \otimes \cE) \mid \cE \in
    \M(\vv)\}\]
  Then, we consider the open dense subset $U$ of $\M(\vv)$
  where equality is attained.
  We get a locally free sheaf $\cV$ over $U$ and a projective
  bundle $\PP(\cV) \to U$, a point of $\PP(\cV)$ being
  given by a pair  $(\cE,[\xi])$, where $[\xi]$ is the
  proportionality class of $\xi \in \Ext^1(\cE,\cS) \setminus
  \{0\}$.
  The element $\xi$ represents a nonsplitting extension:
  \begin{equation} \label{extension E S}
    0 \to \cS \xrightarrow{f} \cF \to \cE \to 0.
  \end{equation}
  where the map $f$ is induced by the extension.
  We get the following dimension count:
  \[
    \dim (\PP(\cV))=\dim(\M(\vv))+h^1(\cS^\vee \otimes \cE)=\vv^2+a+1.
  \]

  Note that $\Ext^2(\cS,\cE)^\vee \simeq \Hom(\cE,\cS)=0$ by stability
  of $\cE$ and $\cS$.
  Also, applying $\Hom(\cS,-)$ to \eqref{extension E S} we have
  \[
    0 \to \End(\cS) \to \Hom(\cS,\cF) \to \Hom(\cS,\cE) \to 0.
  \]
  Note that, using $\chi(\cS^\vee \otimes\cE)=-\vv \cdot \s$, we
  get $\dim(\Hom(\cS,\cE))= a - \vv \cdot \s $. Therefore we have
  \[
    \dim(\Hom(\cS,\cF))=\dim(\Hom(\cS,\cE))+1= a - \vv \cdot \s 
    + 1.
  \]
  
  Assume \correct{for}
  contradiction that $\cF$ is simple and consider its class in the moduli
  space $\Spl(\vv+\s)$ of simple sheaves of Mukai vector $\vv+\s$.
  Here, 
  $(\cE,[\xi])$ determines the point
  $(\cF,[f])$, where $[f]$ is the proportionality class of 
  $f \in \Hom(\cS,\cF) \setminus \{0\}$. The space of these pairs has dimension
  \[
    \dim(\Spl(\vv+\s))+\dim(\PP(\Hom(\cS,\cF)) = (\vv+\s)^2+2+a-\vv
    \cdot \s = \vv^2+\vv\cdot\s+a.
  \]

  However, from the pair $(\cF,[f])$ we recover $\cE$ as $\cE \simeq \coker(f)$ and the
  class $[\xi]$ as the extension given by \eqref{extension E S},
  therefore we must have
  \[
    \vv^2+a+1 \le \vv^2+\vv \cdot\s+a.
  \]
  But this \correct{contradicts} the assumption $\vv\cdot\s \leq 0$.
\end{proof}
\begin{lemma}\label{slopestable}
  Let $s_2\in\N$. 
  Let $X$ be a K3
  surface having \correct{$\Pic(X)=\ZZ \cdot H\oplus^{\bot} N$} with $H^2=4s_2-2$ and
  $N$ containing no effective divisor. We consider $\cS$, a slope-stable bundle of Mukai vector $\s=(2,H,s_2)$. 
  Let $\cE$ be a generic slope-stable sheaf of Mukai vector $(2k+1,(k+1)H,v_2)$ with $v_2\in\N$.
  Consider a non-trivial extension
  \[0 \to \cS \to \cF \to \cE \to 0.\]
  Then $\cF$ is slope-stable.
  \end{lemma}

\begin{proof}
  Let $\cK$ be a stable destabilizing subsheaf of $\cF$.
  We want to check that assuming the existence of $\cK$ we are led to a
  contradiction. As in Lemma \ref{no H1}, write $\cK'$ and $\cK''$ for the subsheaves of $\cS$ and $\cE$ induced
  by the subsheaf $\cK$ of $\cF$.

  We start by proving the result for $k=0$ which is slightly different. In this case $\cE=\cI_Z(H)$ with $Z$ a generic 0-dimensional scheme of the appropriate length. If $\cK''=0$, then the inclusion of $\cK' \simeq \cK$ into $\cS$
  would either give a rank-1 destabilizing subsheaf, or a
  subsheaf of rank $2$ with 
	$\deg_H(\cK')=k'$
	and $k' \ge 2$.
  However both cases are impossible. \par
  Hence $\cK'' \ne 0$, so actually $\cK''$ must be of the form
  $\cI_{Z'}(H)$ for some $0$-dimensional subscheme $Z'$
  containing $Z$, hence $\cQ''=\cI_Z(H)/\cK''$ is a
  $0$-dimensional sheaf. \par
  Now if $\cK'=0$, since $\cQ''$ is 0-dimensional \correct{and} $\cS$ is
  locally free, we get
  $\Ext^1(\cQ'',\cS)=0$, hence $\cQ''$
  is a direct summand of $\cQ=\cF/\cK$, against the assumption
  that $\cK$ is saturated. \par
  However if $\cK'\ne 0$ then $\rk(\cK')=1$, so that in order for $\cK$ to stabilise $\cF$ we
  would need 
		$\deg_H(\cK')=k'$
	with $k' \ge 2$, which is
  impossible by the stability of $\cS$. \par
  
  \smallskip

  Now, we consider the case $k\geq1$, so $\rk\cE\geq2$. Therefore, in this case, for $\cE$ generic, we can assume that $\cE$ is locally free; then $\cF$ and $\cK$ are also locally free.
  Start by writing
  \[
    \mu(\cE)=\frac{k+1}{2k+1}=\frac 12 + \frac 1{4k+2}, \qquad
    \mu(\cF)=\frac{k+2}{2k+3}=\frac 12 + \frac 1{4k+6}.
  \]
    Since $\cK$ destabilizes $\cF$,
  we have:
  \[
    \mu(\cK) \ge \mu(\cF)=\frac 12 + \frac 1{4k+6}, \qquad \rk(\cK) < 2k+3.
  \]

  \begin{itemize}
  \item Assume $\cK''=0$. Then $\cK=\cK'$, so $\rk(\cK) \in \{1,2\}$.
    If $\rk(\cK) = 1$ then by stability of $\cS$ we have 
				$\deg_H(\cK)=a$
		with
    $a \le 0$, so $\cK$ does not destabilize $\cF$.
    If $\rk(\cK) = 2$ then we have 
		$\deg_H(\cK)=a$
				with
    $a \le 1$, so $\cK$ does not destabilize $\cF$.
    This case is thus excluded and we assume from now on $\cK'' \ne 0$.
  \item Let us prove $\mu(\cK'') < \mu(\cE)$.
    By stability of $\cE$, we always have
    $\mu(\cK'') \le \mu(\cE)$ and strict inequality holds if  
    $\rk(\cK'') < \rk(\cE)=2k+1$. 
    \begin{itemize}
    \item If $\cK' = 0$, then $\cK=\cK''$ is a locally free subsheaf
      of $\cE$.
      Note that $\mu(\cK'')=\mu(\cK) < \mu(\cE)$ holds unless $\cK$ and
      $\cE$ have \correct{the} same  
      rank and determinant.
      But in this case,
      since $\cK''$ is a locally free subsheaf of $\cE$ with the same slope
      as $\cE$ and $\cE$ is stable, the inclusion $\cK \to \cE$ is an
      isomorphism and therefore the sequence defining $\cF$ splits, a
      contradiction.
    \item Assume thus $\cK' \ne 0$. Again, $\mu(\cK'')< \mu(\cE)$ unless
      $\rk(\cK'') = \rk(\cE)=2k+1$.  However in this case $\cK'$ must be
      of rank $1$, since $\rk(\cK)<\rk(\cF)$. Then, by \correct{the} stability of
      $\cS$, we get 
						$\deg_H(\cK')=a$
			with
      $a \le 0$, hence
      \[
        \mu(\cK) \le \frac {k+1}{2k+2}.
      \]
      But we immediately check
      \[
        \frac {k+1}{2k+2} <       \frac {k+2}{2k+3} = \mu(\cF).
      \]
      Therefore, we have $\mu(\cK) < \mu(\cF)$, a contradiction.
    \end{itemize}
  \item We proved $\mu(\cK'') < \mu(\cE)$. Also, $\cK$
    is stable and $\cK''$ is a torsion-free quotient of $\cK$
    so $\mu(\cK) \le \mu(\cK'')$, with strict inequality whenever
    $\cK' \ne 0$.
    Hence
    \begin{equation}
      \mu(\cF) \le \mu(\cK) \le \mu(\cK'') < \mu(\cE).
      \label{slopes}
    \end{equation}
    In other words
    \[
      \frac 12+\frac 1{4k+6} \le \mu(\cK) \le \mu(\cK'')< \frac 12+\frac 1{4k+2} 
    \]
    Hence:
    \[
      \frac 12+\frac 1{4k+6} \le \frac{c}{\rk\cK)}< \frac 12+\frac 1{4k+2},
    \]
    with 
		$\deg_H(\cK)=c$.
		    This gives:
    $$\frac{\rk(\cK)}{2}<c<\frac{\rk(\cK)}{2}+1.$$
    We obtain that the only solution is $\rk(\cK)=2r+1$ and $c=r+1$, with $0\leq r\leq k$.
    However, the sequence $\left(\frac{n+1}{2n+1}\right)$ is strictly decreasing. 
    So $\mu(\cK)=\frac{r+1}{2r+1}\geq\frac{k+1}{2k+1}$, which contradicts (\ref{slopes}). 
    Therefore there is no possibility for the slope of $\cK$, this sheaf cannot exist and $\cF$ is then stable.
    
  \end{itemize}
\end{proof}

\begin{proof}[Proof of Theorem \ref{mainth}] 
  Note that $\s\cdot\vv=-(2k+1)$.
  So, as follows from \cite[Lemma 2.6]{yoshioka:crelle}, or by the two
  previous lemmas, we have $\Ext^1(\cS,\cE)=0$ for a generic $\cE\in \M(\vv)$. 
  Since $\cE$ and $\cS$ are stable, we have $\Ext^2(\cS,\cE)=0$ and by Riemann--Roch $\dim (\Hom(\cS,\cE))=2k+1\neq0$.
  Therefore we obtain our result from Lemma \ref{rk1v} and Propositions \ref{torsionfree} and \ref{stability2}.
\end{proof}
\begin{rmk}\label{Hilbremark}
  Note that when $k=0$, the involution $\Phi_{\cS,1}$ acts on
  $X^{[n]}$ with $n=\frac{g+2}{4}$. For $g=10$, this agrees with
  the construction of \cite{beri2022birational}.
  For all $g \equiv 2$ modulo $4$, one can describe this involution by
  taking the residual subscheme with respect to a section of
  $\cS$ vanishing along an element of $X^{[n]}$.
  See the end of \S \ref{examples} for more details on
  this.

\end{rmk}

\begin{rmk}
To go further and construct more involutions, we could also consider a construction more complicated than the one exposed in Definition \ref{maindefi} and Section \ref{introaction}. An idea that could be explored would be to compose several spherical twists.
\end{rmk}

\section{Properties of the involutions}

Here we discuss some basic properties of the involutions
constructed in this paper. After reviewing the connection with
several well-known examples, we provide some results about the
nature of our involutions and 
their
induced action on the cohomology of the moduli space $\M(\vv)$.

\subsection{Examples}\label{examples}

Let us point out how our construction gives a general
framework to many examples of involutions constructed so far
on moduli spaces.

\subsubsection{A quick review of involutions defined by $\cO_X$}

Here we just give a brief account of some well-known involutions and
observe that they can \correct{be} described as $\Phi=\Phi_{\cO_X,0}$.

\begin{ex}\label{Beauvilleinv}
  For $r=1$, let us describe the cases $g=2$ and $g=3$ allowed by the
  inequalities of Theorem \ref{inv}. 
  The moduli space
  $\M(\vv)=\M(1,H,1)$ is identified with the Hilbert scheme of subschemes
  $Z$ of length $g-1$ of $X$. Note that an element of $\M(\vv)$ is of the
  form $\cI_Z(H)$. Taking double dual we easily see
  $\cExt^1(\cI_Z(H),\cO_X) \simeq \omega_Z$ and for subschemes $Z\subset X$ of
  length $1$ or $2$, $\omega_Z$ is identified with $\cO_Z$.

  For $g=3$, $X \subset \PP^3$ is a quartic surface. Let us see that
  $\Phi$ agrees with the Beauville involution sending a length-$2$
  subscheme $Z$ of $X$ to its residual in $X$ with respect to the line
  in $\PP^3$ generated by
  $Z$. For any such $Z \subset X$, we have $h^0(\cI_Z(H))=2$. The base locus in $X$ of the pencil
  of planes through the line $L \subset \PP^3$ spanned by $Z$ is the residual
  subscheme $Z'$ with respect to $L \cap X$. 
    In other words, 
  the cokernel of $\be_{\cI_Z(H)}$ is $\cO_{Z'}$, while
  clearly the kernel of $\be_{\cI_Z(H)}$, being a reflexive sheaf of rank $1$, is just
  $\cO_X(-H)$. We summarize this in the exact sequence:
  \[
    0 \to \cO_X(-H) \to \cO_X^{\oplus 2} \to \cI_Z(H) \to \cO_{Z'} \to 0.
  \]
  The image of the middle map in the above sequence is just $\cI_{\bar
    Z}(H)$, where we put $\bar Z=L \cap X$.
  Dualizing the previous display, we see that \eqref{eq: the map phi} becomes:
  \[
    0 \to \cO_X(-H) \to \cO_X^{\oplus 2} \to \Phi(\cI_Z(H)) \to
    \cO_Z \to 0.
  \]
  It is now clear that $\Phi(\cI_Z(H)) \simeq \cI_{Z'}(H)$.
  
  \smallskip
  By the same argument we see that, for $g=2$, $\Phi$ is just the
  involution exchanging the sheets of the double cover $\pi: X \to \PP^2$,
  defined by sending a point $p$ to $\cI_p(H)$ and then via $\Phi$ to
  $\cI_{p'}(H)$, where $p'$ is the conjugate point of $p$. Indeed, $\cO_{p'}$
  appears as cokernel of the map $\cO_X^{\oplus 2} \to \cI_p(H)$ given
  by the pencil of lines through $\pi(p)$.
\end{ex}

\begin{ex}
  For $r=2$, we have biregular Markman-O'Grady involutions for $g=4,5,6,7,8$. For $g=5$, the surface $X$ is the
  intersection of three quadrics in $\PP^5$ and the moduli space
  $\M(\vv)$ is a K3 surface of genus $2$. More specifically, $\M(\vv)$ can
  be seen as the double cover $\pi: \M(\vv) \to \PP^2$, where the plane
  $\PP^2$ parametrizes the 
  web of quadrics whose base locus is $X$. The branch locus of $\pi$ is the
  discriminant sextic of the web. Pairs of conjugate points $p,p'$ for
  this cover correspond to pairs of spinor bundles $\cG$, $\cG'$ defined over a
  non-singular $4$-dimensional quadric $Q$ of the web. These bundles
  fit into:
  \[
    0 \to \cG^\vee \to \cO_Q^{\oplus 4} \to \cG' \to 0.
  \]
  After restriction to $X$, the above sequence shows that
  $\cG'|_X \simeq \Phi(\cG|_X)$ so $\Phi$ in again the involution exchanging sheets
  of $\pi: \M(\vv) \to \PP^2$.
  
  For $g=7$, the involution $\Phi$ was described at the level of Fano
  threefolds in \cite{brafa1}.
\end{ex}

\begin{ex}
  For $r=3$, in case $g=10$ we have that $\M(\vv)=\M(3,H,3)$ is a K3
  surface of genus $2$. In this case, $X$ is a \correct{codimension $3$}
  linear section of the $5$-dimensional variety $\Sigma$ in $\PP^{13}$
  homogeneous under the exceptional complex Lie group $G_2$. The
  double cover $\pi$ of the plane parametrizing the hyperplanes in
  $\PP^{13}$ whose base locus is $X$ and ramified along the sextic
  curve obtained by cutting the dual of $\Sigma$ along this plane is
  identified with $\M(\vv)$.
  A point of $\M(\vv)$ is a rank-3 vector bundle $\cE$ providing an
  embedding of $X$ in $\GG(3,6)$, cf. \cite{kapustka-ranestad}. The tautological bundles
  $\cG$, $\cG'$ of rank $3$ defined $\GG(3,6)$ fit into:
  \[
    0 \to \cG^\vee \to \cO_{\GG(3,6)}^{\oplus 6} \to \cG' \to 0.
  \]
  Again, restricting to $X$ \correct{in} the above sequence we get that
  $\cG|_X' \simeq \Phi(\cG|_X)$. Exchanging the sheets of $\pi$
  corresponds to exchanging the 
  restricted tautological bundles on $\GG(3,6)$.
\end{ex}

  \subsubsection{Involutions on Hilbert schemes from spherical bundles
  of rank 2}

Here we mention that, if $S$ is a K3 surface of genus $g$ with $g
\equiv 2$ modulo $4$, then, setting $k=0$, the birational involution in Theorem
\ref{mainth} acts on $X^{[n]}$, with $n=(g+2)/4$. Note that $X$
carries a spherical bundle $\cS$ of rank $2$ with $c_1(\cS)=H$ and
$c_2(\cS)=2n$. We get the following result. As we have been
informed, a similar construction appears also in
\cite{beri-manivel:more-birational}.

\begin{prop}
  The involution $\Phi$ sends a generic subscheme $Z \in X^{[n]}$ to its
  residual subscheme with respect to the zero-locus of the only section of $\cS$ vanishing at $Z$.
\end{prop}
  
  \begin{proof}
    Note that $c_2(\cS)=2n$. Let $Z\in X^{[n]}$ be a generic element. According to Riemann--Roch
  and since $\HH^1(\cS \otimes \cI_Z)=0$ for generic $Z$ by
  \cite[Lemma 2.6]{yoshioka:crelle}, 
    we know that $\Hom(\cS,\cI_{Z}(H))=\C$ and $\Ext^i(\cS,\cI_{Z}(H))=0$ for all $i>0$. Let $s\in \Hom(\cS,\cI_{Z}(H))=\HH^0(\cS^\vee\otimes \cI_{Z}(H))=\HH^0(\cS\otimes \cI_{Z})$.
  So $Z\subset \VV(s)$. We obtain the following diagram of exact sequences with $\be$ the evaluation map:
  $$\xymatrix@R10pt{0\ar[r]&\ar@{=}[d]\cO_X\ar[r]^s&\cS\ar@{=}[d]\ar[r]&\cI_{\VV(s)}(H)\ar@{_{(}->}[d]\ar[r]&0
    \\
    0\ar[r]&\cO_X\ar[r]^s&\cS\ar[r]^-{\be} &\cI_{Z}(H)\ar[r] &\coker(\be). }$$
  So $\coker (\be)$ is also the cokernel of the map $\cI_{\VV(s)}(H)\hookrightarrow \cI_{Z}(H)$.
  Therefore $\coker(\be)$ is a torsion sheaf supported
  on a subscheme of length $\ell(\VV(s))-\ell(Z)=n$. We set $Z'\coloneqq\Supp (\coker(\be))$.
  Now, we consider the dual of the evaluation map by $\cO_X(H)$ and we obtain the following commutative diagram:
  $$\xymatrix@R8pt{0\ar[r]&\cO_X\ar[r]^-{\be^{\vee}}&\cS\ar[d]\ar[r]^-{s^{\vee}}&\cO_X(H)\ar[r]&\cExt^2(\coker(\be),\cO_X)\ar[r]&0.\\
    &0\ar[r]&\Phi_{\cS,1}(\cI_Z(H))\ar[ur]& & &}$$

  The sheaf $\cExt^2(\coker(\be),\cO_X)$ is a torsion sheaf supported on $Z'$, so we have $\cExt^2(\coker(\be),\cO_X)=\cO_{Z'}$.
  This shows that $\Phi_{\cS,1}(\cI_Z(H))=\cI_{Z'}(H)$.
  The subscheme $Z'$ described here is the residual with respect to
  the only section $s$ of $\cS$ vanishing at $Z$.
  \end{proof}

\begin{ex}\label{BeriExample}
  The involution studied in \cite{beri2022birational, beri-manivel:more-birational} by Beri and Manivel
  is the same as the involution of Theorem \ref{mainth} with $g=10$
  and $k=0$. Indeed, 
  setting $g=10$ we get $n=3$, so we consider
  $Z=\left\{p_1,p_2,p_3\right\}\in X^{[3]}$, a 
  generic element.
  The bundle $\cS^\vee$ is the restriction of the
  tautological rank 2 bundle on $\G(2,7)$ via the embedding
  $X\hookrightarrow\G(2,7)$. Moreover, we have $\cS=\cS^\vee(H)$ and
  $\HH^0(\cS)=\C^7$.
  For all $s\in \HH^0(\cS)$, we have $\VV(s)=\G(2,6)\cap X\subset \G(2,7)$ and $\VV(s)$ is a subscheme of dimension 0 and of length $c_2(\cS)=6$.
  Hence each point in $X$ can
  be seen as a plane in $\C^7$ and,
  for each $i\in\left\{1,2,3\right\}$, let $P_i$ be the
  plane associated to $p_i$.
  The planes \correct{$P_1,P_2,P_3$ generate} a vector subspace $V_6\subset\C^7$
  of dimension 6. 
  We obtain: $\VV(\widetilde{s})=\G(2,V_6)\subset\G(2,7)$, with
  $\widetilde{s}$ a section of the dual tautological subbundle of
  rank $2$ on $\G(2,7)$.
  We get $s=\widetilde{s}_{|X}\in \HH^0(\cS)$ and 
   $Z\subset \VV(s)=X\cap\G(2,V_6)$ with $\VV(s)$ which is
   0-dimensional of length 6.  
  Beri and Manivel in \cite[Page 8]{beri2022birational} send $Z$ to
  $Z'=\left\{q_1,q_2,q_3\right\}$ with
  $\VV(s)=\left\{p_1,p_2,p_3,q_1,q_2,q_3\right\}$. This is the same as
  our involution since $s\in \HH^0(\cS\otimes \cI_Z)$ and
  $Z'=\{q_1,q_2,q_3\}$ is the residual of $Z=\{p_1,p_2,p_3\}$ with
  respect to the zero-locus of $s$.
\end{ex}
  
\subsection{Anti-symplectic involution}\label{antisympl}
Let $\M(\vv)$ be a compact moduli space of stable sheaves on $X$ of dimension $\geq4$. 
We know from O'Grady \cite{ogrady:Hodge} that there is a Hodge isometry: 
\begin{equation}
  \theta:\vv^{\bot}\subset \HH^*(X,\Z)\rightarrow \HH^2(\M(\vv),\Z),
  \label{Ogradyeq}
\end{equation}
with $\HH^*(X,\Z)$ endowed with the Mukai pairing and $\HH^2(\M(\vv),\Z)$ with the Beauville--Bogomolov form.
The following result has been proved by O'Grady.
\begin{prop}[See {\cite[Proposition 4.14]{OGrady}}]\label{OgradyProp}
  Let $(X,H)$ be a polarized K3 surface such that \correct{$\Pic(X)=\ZZ \cdot H\oplus^{\bot} N$},
 with 
  $H^2=2g-2$ and $N$ not containing effective divisors. Let $\vv=(r,H,r)$ be a Mukai vector with $r\geq2$ and $\dim \M(\vv)\geq4$.
  Let $h_{\vv}\coloneqq\theta(1,0,-1)$ and $R_{h_{\vv}}(\alpha)=-\alpha+B_{\M(\vv)}(h_{\vv},\alpha)h_{\vv}$,
  with $B_{\M(\vv)}$ the Beauville--Bogomolov form on $\M(\vv)$.
  The involution $\Phi_{\cO_X,0}$ on $\M(\vv)$ is anti-symplectic. 
  Moreover \correct{the action of $\Phi_{\cO_X,0}$ on the cohomology $H^2(M(\vv,\Z)$ is given by}: $$\Phi_{\cO_X,0}^*=R_{h_{\vv}}.$$
\end{prop}
\begin{rmk}
  O'Grady also gives the cohomology action of the Beauville involution in \cite[Proposition 4.21]{OGrady}.
\end{rmk}
Using the same ideas as O'Grady, we are going to provide a generalization of Proposition \ref{OgradyProp} to the involution $\Phi_{\cS,1}$.
We consider a Mukai vector $(v_0,v_1 H,v_2)$ with $v_2=(g-1)v_1-\frac{v_0}{2}(\frac{g}{2}-1)$. 
Moreover, we assume that $\dim \M(\vv)\geq4$. As before $\cS$ is the stable bundle on $X$ with Mukai vector $(2,H,\frac{g}{2})$. Let $U(\vv)\subset \M(\vv)$ be the set where $\Phi_{\cS,1}$ is well defined.
Let $d_{\vv}\coloneqq\theta(2,H,\frac{g}{2}-1)$ and $R_{d_{\vv}}(\alpha)=-\alpha+B_{\M(\vv)}(d_{\vv},\alpha)d_{\vv}$, with $B_{\M(\vv)}$ the Beauville--Bogomolov form.
\begin{lemma}\label{CohomologyLemma}
We assume that $U(\vv)$ is a non-empty Zariski open set in $\M(\vv)$.
  Let $\iota_{\vv}:U(\vv)\hookrightarrow \M(\vv)$ be the embedding. \correct{The action of $\Phi_{\cS,1}$ on $H^2(M(\vv,\Z))$ satisfies:}
  $$\iota_{\vv}^*\circ\Phi_{\cS,1}^*=\iota_{\vv}^*\circ R_{d_{\vv}}.$$
\end{lemma}
\begin{proof}
  \correct{Let $\cR$ be} a quasi-family of sheaves on $X$ parametrized by $\M(\vv)$, where $\vv$ is the Mukai vector. 
  We also set $\sigma:X\times \M(\vv)\rightarrow X$ and
  $\pi:X\times \M(\vv)\rightarrow \M(\vv)$ for the natural projections. We set $\HH^4(X,\Z)=\Z \eta$.
  We know (see for instance \cite[4.2.8]{OGrady}) that:
  $$\theta(\alpha)=\theta_{\cR}(\alpha)=\frac{1}{\rho(\cR)}\pi_*\left[\ch(\cR)(1+\sigma^*(\eta))\sigma^*(\alpha^\vee)\right]_6,$$
  with $\rho(\cR)$ an integer such that $\cR_{|X\times\left\{t\right\}}=F^{\rho(\cR)}$ for all $t\in \M(\vv)$ and with $F$ a sheaf with Mukai vector $\vv$. 
  We want to compute: $$\iota_{\vv}^*\circ\Phi_{\cS,1}^*\theta_{\cR}(\alpha)=\theta_{(\id\times\Phi)^*\cR_{|X\times U(\vv)}}(\alpha).$$
  For simplicity in the notation, we are going to denote $\cR_{|X\times U(\vv)}$ also by $\cR$.
  The first step is to understand $(\id\times\Phi)^*\cR$. 
  It is given by the complex $B\coloneqq\RcHom(A,\sigma^*(\cO_X(H))),$
  with $$A\coloneqq\pi^*(\pi_*(\cR\otimes\sigma^*(\cS^\vee)))\otimes\sigma^*\cS\rightarrow\cR.$$
    We have:
  \begin{align*}
    \ch(B)=\ch(\RcHom(A,\sigma^*(\cO_X(H))))&=\ch(\RcHom(A,\sigma^*(\cO_X)))\ch(\sigma^*(\cO_X(H)))=\\
                                            &=\ch(A)^\vee \ch(\sigma^*(\cO_X(H))).
  \end{align*}
  Since $\Phi_{\cS,1}$ is well defined on $U(\vv)$, $B$ is a quasi-family of sheaves and we can consider:
  \begin{align*}
    \theta_{B}(\alpha)&=\pi_*\left[\ch(A)^\vee(1+\sigma^*(\eta))\sigma^*(\alpha^\vee \ch(\cO_X(H)))\right]_6\\
                      &=\pi_*\left[\ch(A)^\vee(1+\sigma^*(\eta))\sigma^*(\alpha
                        \ch(\cO_X(-H)))^\vee\right]_6.
  \end{align*}
  We refer to Section \ref{notata} for the computation of $\alpha\ch(\cO_X(-H))$.
  We set $\beta\coloneqq (\alpha \ch(\cO_X(-H)))^\vee$.
  We obtain:
  \begin{align*}
    \theta_{B}(\alpha)&=-\pi_*\left[\ch(A)(1+\sigma^*(\eta))\sigma^*(\beta^\vee)\right]_6\\
                      &=\pi_*\left[\ch(\cR)(1+\sigma^*(\eta))\sigma^*(\beta^\vee)\right]_6\\
                      &-\pi_*\left[\ch(\pi^*(\pi_*(\cR\otimes\sigma^*(\cS^\vee)))\otimes\sigma^*(\cS))(1+\sigma^*(\eta)\sigma^*(\beta^\vee)\right]_6.\ \ \ \ \ \ \ \ \ \ \ \ \ \ \ \correct{(\clubsuit)}
											\end{align*}
  The first term of the last equality is just $\theta_{\cR}(\beta)$. It remains to compute the second term; we set:
  $$(\heartsuit)\coloneqq\pi_*\left[\ch(\pi^*(\pi_*(\cR\otimes\sigma^*(\cS^\vee)))\otimes\sigma^*\cS)(1+\sigma^*(\eta)\sigma^*(\beta^\vee)\right]_6.$$
  By projection formula, we have:
  \begin{align*}
    (\heartsuit)&=\left[\ch(\pi_*(\cR\otimes\sigma^*(\cS^\vee)))\pi_*(\ch(\sigma^*(\cS))(1+\sigma^*(\eta))\sigma^*(\beta^\vee))\right]_6\\
       &=c_1(\pi_*(\cR\otimes\sigma^*(\cS^\vee)))\pi_*\left[\sigma^*(\ch(\cS)(1+\sigma(\eta)))\beta^\vee\right]_4.
  \end{align*}
  We set $(\diamondsuit)\coloneqq c_1(\pi_*(\cR\otimes\sigma^*(\cS^\vee)))$ and $(\spadesuit)\coloneqq\pi_*\left[\sigma^*(\ch(\cS)(1+\sigma(\eta)))\beta^\vee\right]_4$.
  We first compute $(\diamondsuit)$; according to Grothendieck--Riemann--Roch theorem, we have:
  \begin{align*}
    (\diamondsuit)&=\pi_*\left[\ch(\cR\otimes\sigma^*(\cS^\vee))\sigma^*(1+2\eta)\right]_6\\
        &=\pi_*\left[\ch(\cR)\sigma^*(\ch(\cS^\vee)(1+2\eta))\right]_6.
  \end{align*}
  However, we have $\ch(\cS^\vee)=(2,-H,\frac{g}{2}-2)$. So: $\ch(\cS^\vee)(1+2\eta)=(2,-H,\frac{g}{2}+2)=(2,-H,\frac{g}{2})(1+\eta)=\vv(\cS)^\vee (1+\eta)$. 
  So:
  $$(\diamondsuit)=\theta_{\cR}(\vv(\cS)).$$
  It remains to compute $(\spadesuit)$. We set $\beta=(\beta_0,\beta_1,\beta_2)$.
  \begin{align*}
    (\spadesuit)&=\pi_*\left[\sigma^*\left(\left(2,H,\frac{g}{2}-2\right)(1+\eta)(\beta_0,-\beta_1,\beta_2)\right)\right]_4\\
         &=\pi_*\left[\sigma^*\left(\left(2,H,\frac{g}{2}\right)(\beta_0,-\beta_1,\beta_2)\right)\right]_4\\
         &=2\beta_2-(2g-2)\beta_1+\frac{g}{2}\beta_0\\
         &=-\left(2,H,\frac{g}{2}\right)\cdot\beta\\
         &=-\vv(\cS)\cdot\beta.
  \end{align*}
	\correct{We replace ($\diamondsuit$) and ($\spadesuit$) in ($\heartsuit$) and then ($\heartsuit$) in ($\clubsuit$) to obtain:}
  $$\iota_{\vv}^*\Phi^*\theta_{\cR}(\alpha)=\iota_{\vv}^*\left[\theta_{\cR}(\beta)+\left(\vv(\cS)\cdot\beta\right)\theta_{\cR}(\vv(\cS))\right].$$
  We set $R_{\cS}(x)=x+\left(\vv(\cS)\cdot x\right) \vv(\cS)$, $D$ the dual and $T$ the tensorization by $\cO_X(-H)$ of Mukai vectors. We have shown that:
	\begin{equation}
  \iota_{\vv}^*\Phi^*\theta_{\cR}(\alpha)=\iota_{\vv}^*\circ\theta_{\cR}\circ R_{\cS}\circ D \circ T(\alpha).
	\label{action2}
	\end{equation}
  Then it remains to show that, 		for all $x\in
  \vv^{\bot}$, 
  $$R_{\cS}\circ D \circ T(x)=-x+\left[\left(2,H,\frac{g}{2}-1\right)\cdot x\right]\left(2,H,\frac{g}{2}-1\right),$$
  It is true for all elements in $(\HH^0(X,\Z)\oplus \Z H
  \oplus \HH^4(X,\Z))^{\bot}$ and we can verify easily
  that it is also true for a basis of
  $\left(\HH^0(X,\Z)\oplus \Z H \oplus
    \HH^4(X,\Z)\right)\cap \vv^{\bot}$, for instance
  $((2,H,\frac{g}{2}-1),(v_0,0,-v_2))$.
\end{proof}

\begin{rmk}\label{actionrk}
A priori the set $\M(\vv)\smallsetminus U(\vv)$ could have codimension
1. To be more precise, the involution $\Phi$ can always be extended to
a bimeromophism regular in codimension 2 \correct{(see for instance \cite[Lemma 3.2]{lol2})}, but $\Phi$ as constructed in
Section \ref{31construction} may be defined only in codimension 1. 
 Therefore, if we also denote by $\Phi$ the extension of $\Phi$ to a
 regular involution in codimension 2, the action of $\Phi^*$ on
 $\HH^2(\M(\vv),\Z)$ is well defined but could differ from the one
 found in Lemma \ref{CohomologyLemma}. 
\end{rmk}

\begin{rmk}
If $\M(\vv)\smallsetminus U(\vv)$ has codimension 2 then according to
\eqref{action2} the action of $\Phi^*$ on $\HH^2(\M(\vv),\Z)$  is given
by : 
$$\Phi^*=\theta \circ R_{\cS}\circ D \circ T\circ \theta^{-1}.$$
Note that there is another action on $\HH^2(\M(\vv),\Z)$ obtained from
the computation on Mukai \correct{vectors}, see Lemma \ref{v}.
It provides the following action :
\begin{align*}
\varphi&=-\left(\theta\circ T^{-1}\circ D\circ
      R_{\cS}\circ \theta^{-1}\right).\\
			&\correct{=-(\Phi^*)^{-1}.}
			\end{align*}
\correct{And since $\Phi^*$ is an involution,} this shows that $\varphi=-\Phi^*$.
\end{rmk}

As explained in Remark \ref{actionrk}, we cannot always determined the
action of $\Phi^*$ on $\HH^2(\M(\vv),\Z)$, however we can deduce from
the previous lemma that $\Phi_{\cS,1}$ is anti-symplectic. 

\begin{prop}\label{anti}
  Let $g\geq2$ even. Let $(X,H)$ be a K3 surface with
  	$H^2=2g-2$ and $\cS$ the stable
  bundle of Mukai vector $(2,H,\frac{g}{2})$. Let
  $\vv=(v_0,v_1 H,(g-1)v_1-\frac{v_0}{2}(\frac{g}{2}-1))$ be a
  Mukai vector (we assume that
  $\frac{v_0}{2}(\frac{g}{2}-1)$ is integral). We set $U(\vv)$ the open set where $\Phi_{\cS,1}$ is well defined. We assume that: 
	\begin{itemize}
	\item
	$\dim \M(\vv)\geq4$;
\item
 $U(\vv)$ is a non-empty Zariski open set in $\M(\vv)$.
	\end{itemize}
  Then, the involution $\Phi_{\cS,1}$ on $\M(\vv)$ is anti-symplectic. 
\end{prop}
\begin{proof}
  Let $T_{\M(\vv)}$ be the transcendental lattice of $\M(\vv)$. From Lemma \ref{CohomologyLemma}, we can prove that $\Phi_{\cS,1|T_{\M(\vv)}}^*=-\id_{T_{\M(\vv)}}$.
  \correct{Indeed,} we have an exact sequence:
  \[\HH^2(\M(\vv),U(\vv),\Z) \xrightarrow{f} \HH^2(\M(\vv),\Z)
    \xrightarrow{\iota_{\vv}^*} \HH^2(U(\vv),\Z).
  \]
  According to \cite[Section 11.1]{Voisin}, $\im (f)\subset \Pic (\M(\vv))$.
  Therefore, $\iota_{\vv}^*$ induces an injection $T_{\M(\vv)}\hookrightarrow \HH^2(U(\vv),\Z)$.
	\correct{Moreover, an element in $T_{\M(\vv)}$ can be written $\theta(0,x,0)$ with $x\in T_{X}$ because $\theta$ respects the Hodge structure.
	In particular, $d_{\vv}$ is orthogonal to $T_{\M(\vv)}$.} So, Lemma \ref{CohomologyLemma} implies \correct{$\Phi_{\cS,1|T_{\M(\vv)}}^*=-\id_{T_{\M(\vv)}}$ which proves} our claim.	
\end{proof}
\begin{rmk}
  If $\dim(\M(\vv))=2$, the involution $\Phi_{\cS,1}$ induces a double cover $\M(\vv)\rightarrow \mathds{P}^2$ and is also anti-symplectic.
\end{rmk}
As a direct consequence, we have the following corollary.
\begin{cor}\label{anticor}
  The involutions defined in Theorems \ref{mainth} are anti-symplectic.
\end{cor}
As explained in Remark \ref{actionrk}, 
 we cannot conclude from Lemma \ref{CohomologyLemma} the action of $\Phi_{\cS,1}^*$ on the cohomology.
 However, using lattice theory, we can show that $\Phi_{\cS,1}^*$ is a reflection through an element of square 2 or $n-1$ or $2(n-1)$ with $2n=\dim(\M(\vv))$.

 \subsection{Action on the cohomology}\label{cohomology}

 In this section, we assume that $\Pic X= \Z\cdot H$ in order to obtain information on $\HH^2(\M(\vv),\Z)^{\Phi_{\cS,1}}$ from lattice considerations.
Let us describe the Picard lattice of $\M(\vv)$, based on \eqref{Ogradyeq}.
\begin{lemma}
  Let $g\geq 2$ even. Let $X$ be a K3 surface such that
  $\Pic(X)=\Z \cdot H$, with $H^2=2g-2$. Let $\vv=(v_0,v_1 H,v_2)$ a Mukai vector with $v_2=(g-1)v_1-\frac{v_0}{2}(\frac{g}{2}-1)$ an integer. 
  We assume that $2n\coloneqq\dim \M(\vv)\geq 4$. We set
  $\delta\coloneqq v_0\wedge v_2$, $d_{\vv}=\theta(2,1,\frac{g}{2}-1)$ and $f_{\vv}\coloneqq\frac{1}{\delta}\theta(v_0,0,-v_2)$.
  Then $(d_{\vv},f_{\vv})$ gives a basis of $\Pic(\M(\vv))$ with bilinear form given by the matrix:
  $$\begin{pmatrix}
    2 &  \frac{2v_2}{\delta}-(\frac{g}{2}-1)\frac{v_0}{\delta}\\
    \frac{2v_2}{\delta}-(\frac{g}{2}-1)\frac{v_0}{\delta}& 2\frac{v_0v_2}{\delta^2}
  \end{pmatrix}.$$
  In \correct{particular}: 
  \begin{equation}
    \discr (\Pic(\M(\vv)))=-\frac{4(g-1)}{\delta^2}(n-1).
    \label{discrPic}
  \end{equation}
\end{lemma}
\begin{proof}
  Since $d_{\vv}$ and $f_{\vv}$ are orthogonal to $\vv$, we have $\Z d_{\vv}\oplus \Z f_{\vv}\subset \Pic(X)$. Moreover, $\Z d_{\vv}\oplus \Z f_{\vv}$ is a primitive sublattice of rank two, so $\Pic(X)=\Z d_{\vv}\oplus \Z f_{\vv}$. The bilinear matrix is obtained because $\theta$ is an isometry. Finally, since 
	\begin{equation}
	n-1=(g-1)v_1^2-v_0v_2, 
	\label{dim}
	\end{equation}
	we obtain the given discriminant. 
\end{proof}
From the two previous lemmas, we deduce the cohomology action of $\Phi_{\cS,1}$.
\begin{prop}\label{action}
  Let $g\in\N$ with $g\equiv2\mod 4$ and $k\in \N$. Let $X$ be a K3 surface such that
  $\Pic(X)=\Z \cdot H$, with $H^2=2g-2$. Let $\vv=(2k+1,
	(k+1) H,v_2)$ be a Mukai vector with $v_2=(g-1)(k+1)-\frac{2k+1}{2}(\frac{g}{2}-1)$. 
    We set $\delta=v_0\wedge v_2$.
  We assume that $2n\coloneqq\dim (\M(\vv))\geq 4$. 
  
  \begin{enumerate}[label=\arabic*.]
    \item
    For $n>2$, if $\HH^2(\M(\vv),\Z)^{\Phi_{\cS,1}}\simeq (2(n-1))$
    then $\delta=1$ and $\frac{-1}{g-1}$ is a square in $\Z/(n-1)\Z$.
  \item \label{P2}
    For even $n>2$, if $\HH^2(\M(\vv),\Z)^{\Phi_{\cS,1}}\simeq
    (n-1)$ then $\delta=1$ and $\frac{-1}{2(g-1)}$ is a square in
    $\Z/(n-1)\Z$.
  \item
    Otherwise $\HH^2(\M(\vv),\Z)^{\Phi_{\cS,1}}\simeq (2)$.
  \end{enumerate}
\end{prop}

About \ref{P2}, note that $\frac{-1}{g-1}$ is well defined in $\Z/(n-1)\Z)$, because $\delta=(n-1)\wedge(g-1)=1$.

\begin{proof}
Before starting the proof note that $\delta=(n-1)\wedge(g-1)$. 
  	This is a direct consequence of the two equations
  $v_2=(g-1)v_1-\frac{v_0}{2}(\frac{g}{2}-1)$ and
  $n-1=(g-1)v_1^2-v_0v_2$. 
	
  We set $T\coloneqq\HH^2(\M(\vv),\ZZ)^{\Phi_{\cS,1}}$.
  By Proposition \ref{anti}, the involution $\Phi$ is anti-symplectic. So:
  $$
  T\subset \Pic (\M(\vv)).
  $$
  Since $\rk(\Pic(\M(\vv)))=2$, according to \cite[Proposition 1.6]{camerenonsymplectic}, there are four possibilities: 
  \begin{enumerate}[label=(\roman*)]
  \item
    $A_{T}=\Z/2\Z$;
  \item
    $A_{T}=\Z/2\Z\oplus \Z/2(n-1)\Z$;
  \item
    $A_{T}=\Z/2(n-1)\Z$, $n>2$.
  \item
    $A_{T}=\Z/(n-1)\Z$, $n>2$ and $n$ even.
  \end{enumerate}
  First note that all these cases imply that $\rk(T)=1$. Indeed if $\rk(T)=2$, it provides $T=\Pic(\M(\vv))$ and so by (\ref{discrPic}):
	\begin{enumerate}[label=(\roman*)]
	\item
	$\delta^2=2(n-1)(g-1)$;
	\item \label{Pii}
	$\delta^2=g-1$;
	\item \label{Piii}
	$\delta^2=2(g-1)$;
	\item \label{Piv}
	$\delta^2=4(g-1)$.
	\end{enumerate}
However the hypothesis $\dim (\M(\vv))\geq 4$ implies that
$g-1>v_0^2\geq\delta^2$. So all the previous cases are impossible. The
case \ref{Pii} can only occur for $\rk T=2$, so it remains to treat
the cases \ref{Piii} and \ref{Piv} with $\rk T=1$.
  	According to \cite[Proposition 1.6]{camerenonsymplectic}, we have the following subcases:
	\begin{enumerate}[label=(\roman*), start=3]
	\item
  \begin{enumerate}[label=(\alph*)]
	\item 
	There exists a primitive element $x$ of square $2(n-1)$ and divisibility $2(n-1)$.
	\item \label{iiib}
	There exists a primitive element $x$ of square $2(n-1)$ and divisibility $(n-1)$.
	\end{enumerate}
	\item
	There exists a primitive element $x$ of square $(n-1)$ and divisibility $(n-1)$.
	\end{enumerate}
		Let $y$ be a primitive element such that
        $B_{\M(\vv)}(x,y)=0$. According to the divisibility of $x$,
        there are two cases, namely $\Z x\oplus \Z y=\Pic(\M(\vv))$ or 
	$\Z x\oplus \Z y$ has index 2 in $\Pic(\M(\vv))$ in the
        subcase \ref{iiib} of \ref{Piii}.
        So $\discr \left(\Z x\oplus \Z
          y\right)=-\frac{4(g-1)}{\delta^2}(n-1)$ or $\discr \left(\Z
          x\oplus \Z y\right)=-\frac{16(g-1)}{\delta^2}(n-1)$
        again in the
        subcase \ref{iiib} of \ref{Piii}.
        It implies in the different cases that: 
        \begin{enumerate}[label=(\roman*), start=3]
	\item
            \begin{enumerate}[label=(\alph*)]
	\item
$y^2=\frac{-2(g-1)}{\delta^2}$,
	\item
	$y^2=\frac{-8(g-1)}{\delta^2}$,
	\end{enumerate}
	\item
	$y^2=\frac{-4(g-1)}{\delta^2}$.
	\end{enumerate}

	Since $\delta$ is odd, it means that $\delta^2\ |\
        (g-1)$. Hence by \eqref{dim}, $\delta^2\  |\ (n-1)$ and so
        $\delta=1$. 
	Moreover, we can write $d_v=\lambda x+\mu y$ or $2d_v=\lambda
        x+\mu y$,
        in the subcase \ref{iiib} of \ref{Piii}, that is:
        \begin{enumerate}[label=(\roman*), start=3]
	\item
          \begin{enumerate}[label=(\alph*)]
          \item
            $2=\lambda^2(2(n-1))+\mu^2\left(-2(g-1)\right)$;
          \item
            $8=\lambda^2(2(n-1))+\mu^2\left(-8(g-1)\right)$;
          \end{enumerate}
	\item
          $2=\lambda^2(n-1)+\mu^2\left(-4(g-1)\right)$.
        \end{enumerate}

        In the subcase \ref{iiib} of \ref{Piii}, $\lambda$ has to be divisible
by 2; so in case (iii) we obtain that $\frac{-1}{g-1}$ is a square in
$\Z/(n-1)\Z$.
	In case \ref{Piv}, $\lambda$ is also divisible by 2 and we
        obtain that $\frac{-1}{2(g-1)}$ is a square in $\Z/(n-1)\Z$.  
\end{proof}

\bibliography{bibliography}
\bibliographystyle{amsalpha-my}
\end{document}